\documentclass[reqno,10pt,twoside]{amsart} 

\usepackage{amsmath,amsthm,amssymb,amsfonts,mathrsfs,color}
\usepackage{latexsym,esint}
\usepackage{mathrsfs}  
\usepackage[colorinlistoftodos,prependcaption,textsize=tiny]{todonotes}

\newtheorem{theorem}{Theorem}[section]
\newtheorem{proposition}{Proposition}[section]

\newtheorem{lemma}{Lemma}[section]
\newtheorem{corollary}{Corollary}[section]
\newtheorem{definition}{Definition}[section]
\newtheorem{remark}{Remark}[section]

\numberwithin{equation}{section} \numberwithin{theorem}{section}
\numberwithin{proposition}{section} \numberwithin{lemma}{section}
\numberwithin{corollary}{section}
\numberwithin{definition}{section} \numberwithin{remark}{section}

\newcommand{\R}{\mathbf R} 

\newcommand{\wto}{\rightharpoonup}
\newcommand{\wsto}{\stackrel{*}{\rightharpoonup}}
\newcommand{\e}{\varepsilon}

\newcommand{\A}{{\mathcal A}}

\newcommand{\LL}{{\mathcal L}}
\newcommand{\HH}{{\mathcal H}}
\newcommand{\M}{{\mathcal M}}
\newcommand{\C}{{\mathscr C}}

\newcommand{\dd}{\,\mathrm{d}}

\DeclareMathOperator{\tr}{tr}
\DeclareMathOperator{\supp}{\text{supp}}
\let\O=\Omega

\newcommand{\dive}{{\rm div }}

\newcommand{\res}{\mathop{\hbox{\vrule height 7pt width .5pt depth 0pt
\vrule height .5pt width 6pt depth 0pt}}\nolimits}

\def \p{\partial}

\author[J.-F. Babadjian]{Jean-Fran\c cois Babadjian}
\author[V. Millot]{ Vincent Millot}
\author[R. Rodiac]{R\'emy Rodiac}

\address[J.-F. Babadjian]{Universit\'e Paris-Saclay, CNRS,  Laboratoire de math\'ematiques d'Orsay, 91405, Orsay, France.}
\email{jean-francois.babadjian@universite-paris-saclay.fr}

\address[V. Millot]{LAMA, Universit\'e  Paris Est Cr\'eteil, Universit\'e Gustave Eiffel, UPEM, CNRS, F-94010, Cr\'eteil, France}
\email{vincent.millot@u-pec.fr}

\address[R. Rodiac]{Universit\'e Paris-Saclay, CNRS,  Laboratoire de math\'ematiques d'Orsay, 91405, Orsay, France.}
\email{remy.rodiac@universite-paris-saclay.fr}

\title[Critical of points of the Ambrosio-Tortorelli functional]{On the convergence of critical points of the Ambrosio-Tortorelli functional}

\begin{document}


\begin{abstract}
This work is devoted to study the asymptotic behavior of critical points \(\{(u_\e,v_\e)\}_{\e>0}\) of the Ambrosio-Tortorelli functional.  {Under a uniform energy bound assumption,} the usual \(\Gamma\)-convergence theory ensures that \((u_\e,v_\e)\) converges in the \(L^2\)-sense to some \((u_*,1)\) as $\varepsilon\to 0$, where \(u_*\) is a special function of bounded variation.  Assuming further the Ambrosio-Tortorelli energy of \((u_\e,v_\e)\) to converge to the Mumford-Shah energy of \(u_*\), the later is shown to be a critical point with respect to inner variations of the Mumford-Shah functional. As a byproduct, the second inner variation is also shown to pass to the limit. To establish these convergence results,  interior \((\C^\infty\)) regularity and boundary regularity for Dirichlet boundary conditions are first obtained for a fixed parameter \(\e>0\). The asymptotic analysis is then performed by means of varifold theory in the spirit of scalar phase transition problems. 
\end{abstract}

\maketitle

\section{Introduction}

Let $\O \subset \R^N$ be a bounded open set with Lipschitz boundary (\(N \geq 1\)) and $g \in H^{\frac12}(\p \O)$  be  a prescribed Dirichlet boundary data on $\partial\O$. For infinitesimal parameters $\e\to 0$ and $\eta_\e \to 0$ with \( 0<\eta_\e \ll \e\), we consider the {\it Ambrosio-Tortorelli functional} defined by
\begin{equation}\label{predefAT}
AT_\e(u,v):= \int_\O (\eta_\e+v^2) |\nabla u|^2\dd x + \int_\O\left(\e|\nabla v|^2 +\frac{(v-1)^2}{4\e}\right) \dd x\, ,
\end{equation}
for all pairs {$(u,v) \in H^1(\O) \times [H^1(\O)\cap L^\infty(\O)]$ satisfying $(u,v)=(g,1)$ on $\partial\O$}. This functional, originally introduced in \cite{Ambrosio_Tortorelli_1990} 
can be interpreted as a {\it phase-field regularization} of the Mumford-Shah functional
\begin{equation}\label{predefMS}
(u,v) \mapsto 
\begin{cases}
\displaystyle MS(u):=\int_\O |\nabla u|^2\dd x + \HH^{N-1}(J_u)+ \HH^{N-1}\big(\partial\O \cap \{u \neq g\}\big) & \text{if }
\begin{cases}
u \in SBV^2(\O)\,,\\
v=1\text{ in }\O\,,
\end{cases}\\[12pt]
+\infty & \text{otherwise}\,.
\end{cases}
\end{equation}
The Mumford-Shah functional is well known as a theoretical tool to approach image segmentation \cite{Mumford_Shah_1985, Marroquin_Mitter_Poggio_1987,Mumford_Shah_1989}. It is also at the heart of the Francfort-Marigo model in fracture mechanics \cite{Francfort_Marigo_1998}, and  the numerical {implementation} of this model heavily relies on Ambrosio-Tortorelli type functionals \cite{Bourdin_2007}. The use of such phase-field approximation in numerics is usually justified through $\Gamma$-convergence theory. In terms of the functionals defined above, it states that $AT_\e$ $\Gamma$-converges in the $[L^2(\O)]^2$-topology as $\e\to0$  towards the Mumford-Shah functional (see e.g.\ the seminal paper \cite{Ambrosio_Tortorelli_1992})

As a consequence, the fundamental theorem of $\Gamma$-convergence ensures the convergence of global minimizers \( (u_\e,v_\e)\) of $AT_\e$ to $(u,1)$ as $\e\to0$ where \(u \in SBV^2(\O)\) is a global minimizer of $MS$. This result is of course of importance, but it is somehow not fully satisfactory. Beyond the fact that the use of global minimizers in the models mentioned above remains under debate, this convergence result does not really provide a rigorous justification of the numerical simulations based on {the} Ambrosio-Tortorelli {functional}.    
One particular feature of  $AT_\e$ is its lack of convexity due to the nonconvex coupling term $v^2|\nabla u|^2$ with respect to the pair $(u,v)$. This is a high obstacle to reach global minimizers through a numerical method. An idea employed in the context of image segmentation or fracture mechanics consists in performing an alternate minimization algorithm,  see \cite{BFM}. Each iteration of the scheme is well-posed since $AT_\e$ is continuous, coercive, and separately strictly convex. Letting the number of steps going to infinity, the sequence of iterates turns out to converge to a critical point of the energy $AT_\e$  \cite[Theorem 1]{Bourdin_2007}, but this critical point might fail to be a global minimizer. Consequently, the original target of numerically approximating global minimizers of the Mumford-Shah functional might be lost. These issues motivate the question of convergence  as $\e\to0$ of critical points of the Ambrosio-Tortorelli functional and it constitutes the main goal of this article, continuing a task initiated in \cite{Francfort_Le_Serfaty_2009,Le_2010} in dimension $N=1$.  In higher dimension, a fundamental issue in such an analysis is the regularity of critical points of $AT_\e$. It is also of importance for numerics  as the efficiency of the numerical methods crucially rests on it. Here, we fully resolve this last question showing smoothness of arbitrary critical points according to the smoothness of $\partial\O$ and the Dirichlet boundary data.

\medskip

A critical point $(u_\e,v_\e)$ of the Ambrosio-Tortorelli functional is a weak (distributional) solution of the nonlinear elliptic system 
\begin{equation}\label{eq:PDE}
\begin{cases}
\displaystyle -\dive\big((\eta_\e+v_\e^2) \nabla u_\e\big)=0 & \text{ in }\O\,,\\[5pt]
\displaystyle -\e \Delta v_\e+\frac{v_\e-1}{4\e}+v_\e |\nabla u_\e|^2=0 & \text{ in }\O\,,\\[5pt]
(u_\e, v_\e)=(g,1) & \text{ on }\partial\O\,.
\end{cases}
\end{equation}
To be more precise, critical points of $AT_\e$ are defined as follows.

\begin{definition}\label{def:critical_AT}
{\rm Let $\O \subset \R^N$ be a bounded open set {with Lipschitz boundary} and {$g \in H^{\frac12}(\partial\O)$}. A pair 
$${(u_\e,v_\e)\in \mathcal A_g(\O):=\Big\{(u,v) \in H^1(\O) \times [H^1(\O)\cap L^\infty(\O)] : \; (u,v)=(g,1) \text{ on }\partial\O\Big\}}$$
is a {\it critical point of the Ambrosio-Tortorelli functional} if
$$\frac{\dd}{\dd t} \Bigl|_{t=0} AT_\e(u+t\phi,v+t\psi)=0\quad  \text{ for all }(\phi,\psi)\in H^1_0(\O)\times [ H_0^1(\O)\cap L^\infty(\O)]\,,$$
that is
\begin{equation}\label{eq:varformu}
\int_\O (\eta_\e+v_\e^2)\nabla u_\e\cdot\nabla\phi\dd x =  0 \quad \text{ for all }\phi\in H^1_0(\O)\,,
\end{equation}
and 
\begin{equation}\label{eq:varformv}
\e\int_\O \nabla v_\e \cdot\nabla \psi \dd x + \int_\O \left(\frac{v_\e-1}{4\e}+v_\e |\nabla u_\e|^2\right)\psi\dd x = 0 \quad \text{ for all }\psi\in H_0^1(\O)\cap L^\infty(\O)\,.
\end{equation}
By density, test functions $(\phi,\psi)$ in \eqref{eq:varformu}-\eqref{eq:varformv}  can equivalently be chosen in $[\C^\infty_c(\O)]^2$, and  \eqref{eq:PDE}  holds in the sense of distributions in $\O$. 
}
\end{definition}

One may expect that critical points of \(AT_\e\) with uniformly bounded energy converge along some subsequence $\e\to0$ to a limit satisfying some first order criticality  conditions for \(MS\). Unfortunately, the theory of \(\Gamma\)-convergence does not provide convergence of critical points towards critical points of the limiting functional. Even for local minimizers such a result usually fails. We refer to \cite[Remark 4.5]{Jerrard_Sternberg_2009} and \cite[Example 3.5.1]{Braides_2014} for counter-examples. However, it has been proved in some specific examples that critical points do converge to critical points, possibly under the assumption of convergence of critical values. This is the case for the Allen-Cahn  (or Modica-Mortola) functional from phase transitions approximating the $(N-1)$-dimensional area functional \cite{PadillaTonegawa,Hutchinson_Tonegawa_2000,Tonegawa_2002,Tonegawa_2005,Guaraco_2018,Gaspar_Guaraco_2018},  the Ginzburg-Landau functional approximating  the $(N-2)$-dimensional area functional \cite{Ambrosio_Soner_1997,Bethuel_Brezis_Orlandi_2001,Lin_Riviere_2001,Stern_2021,Pigati_Stern_2021}, and   
the Dirichlet energy of manifold valued stationary harmonic maps  \cite{Lin99,Lin_Riviere_2002,LinWang1,LinWang2,LinWang3}. These functionals share many features with $AT_\e$, and we shall 
take advantage of the existing theory  to develop our asymptotic analysis of critical points of $AT_\e$. In particular, we shall make an essential use of both outer and inner variations of the energy, a common approach in all these studies.

\subsection{Outer and inner variations}

Definition \ref{def:critical_AT} is simply saying  that the first outer variation of \(AT_\e\) vanishes at \( (u_\e,v_\e) \in {\mathcal A_g(\O)}\) in any direction \( (\phi,\psi)\). In case of a smooth functional like  \(AT_\e\), outer variations coincide with G\^ateaux differentials. 
For \( (u,v) \in {\mathcal A_g(\O)}\) and \( (\phi,\psi)\in H^1_0(\O)\times [ H_0^1(\O)\cap L^\infty(\O)]\) as before, we introduce the following notation for the first and second outer variations of \(AT_\e\) (see Lemma~\ref{lem:outer-var} for explicit formulas)
\begin{align}
\label{1outvar}\dd {AT_\e} (u,v)[\phi,\psi] := \frac{\dd}{\dd t} \Bigl|_{t=0} AT_\e(u+t\phi,v+t\psi)\,,  \\[5pt]
\label{2outvar}\dd^2AT_\e (u,v)[\phi,\psi]:=\frac{\dd^2}{\dd t^2} \Bigl|_{t=0} AT_\e(u+t\phi,v+t\psi)\,. 
\end{align}

\medskip

Concerning the Mumford-Shah functional, the notion of critical points requires some definition and notation. Before doing so, let us first comment on the functional $MS$ in \eqref{predefMS} we are considering. Contrary to $AT_\e$, the admissible $u$'s for $MS$ are not required to agree with $g$ on $\partial\Omega$ in the sense of traces.  In turn, the additional term $\HH^{N-1}(\partial\O \cap \{u\neq g\})$ in the expression of $MS(u)$ penalizes ``boundary jumps'' where the inner trace of $u$ (still denoted by $u$) differs from $g$.
 The expression $u\neq g$ on $\partial\O$ is also intended in the sense of traces. In the sequel, we shall often use the following compact notation
$$MS(u)=\int_\O |\nabla u|^2\dd x +  \HH^{N-1}(\widehat{J}_u)\,, \quad u \in SBV^2(\O)\,,$$
where $\widehat J_u=J_u \cup (\partial\O\cap \{u\neq g\})$, so that 
$$\widehat J_u=J_{\hat u} \;\text{ with }\;\hat u:=u{\bf 1}_\O + G{\bf 1}_{\R^N \setminus \O}\in SBV^2(\R^N)\,,$$
and {\(G \in H^1(\R^N)\) is an arbitrary extension of \(g\)}.
\medskip

Unlike  \(AT_\e\), the Mumford-Shah functional is not smooth and outer variations must be accordingly defined (see e.g. {\cite[Section 7.4]{Ambrosio_Fusco_Pallara_2000}}). 
Given $u$, $\phi \in SBV^2(\O)$ such that \( \widehat{J}_\phi \subset \widehat{J}_u\),   the first and second outer variations of \(MS\) at \(u\)  in the direction $\phi$ are respectively defined and given by
\begin{align*}
\dd MS(u)[\phi]:= \frac{\dd}{\dd t}\Bigl|_{t=0} MS(u+t\phi)=2\int_\O\nabla u\cdot\nabla\phi\,\dd x\,,\\[5pt]
 \dd^2MS(u)[\phi]:= \frac{\dd^2}{\dd t^2}\Bigl|_{t=0} MS(u+t\phi)=2\int_\O|\nabla\phi|^2\,\dd x\,.
 \end{align*}
 In this definition, the requirement  \(\widehat{J}_\phi \subset \widehat{J}_u\)  ensures the differentiability at $t=0$ of the function $t\mapsto MS(u+t\phi)$ since 
\( \mathcal{H}^{N-1}(\widehat{J}_{u+t\phi})\) remains constantly equal to \( \mathcal{H}^{N-1}(\widehat{J}_{u})\) . As a consequence, these differentials provide  only information on the ``regular part'' of the function $u$, and not on the jump set~\(\widehat{J}_u\). Note also that the second order condition \( \dd^2MS(u)[\phi] \geq 0\) is obviously satisfied at any \(u, \phi\) as above.  On the other hand, the condition \(\widehat{J}_\phi \subset \widehat{J}_u\) also implies that the direction $\phi$ must agree with $g$ on $\partial\Omega\cap\{u=g\}$, in agreement with the notion of Dirichlet boundary condition.

It is clear that outer variations are not sufficient to define a  notion of critical point for~$MS$ since admissible perturbations leave the ``singular part''  \( \mathcal{H}^{N-1}(\widehat{J}_{u})\) unchanged. The way to complement outer variations is to consider  {\it inner variations}, i.e., variations under domain deformations. 
In doing so (up to the boundary), we shall assume that $\partial\O$ is at least of class $\C^2$. 

Given a vector field  $X \in \C^1_c(\R^N;\R^N)$ satisfying  $X\cdot \nu_\O=0$ on $\partial\O$ (here \(\nu_\O\) denotes the outward unit normal field on \(\p \O\)), we consider its flow map $\Phi:\R\times\R^N\to\R^N$, i.e., for every $x\in\R^N$, $t\mapsto\Phi(t,x)$ is defined as the unique solution of the system of ODE's
\begin{equation}\label{eq:la_coulée}
\begin{cases}
\displaystyle \frac{\dd \Phi}{\dd t}(t,x)=X(\Phi(t,x))\,, \\[5pt]
\Phi(0,x)=x\,.
\end{cases}
\end{equation}
According to the standard Cauchy-Lipschitz theory, $\Phi\in \C^1(\R\times\R^N;\R^N)$ is well-defined, and $\{\Phi_t\}_{t\in\R}$ with $\Phi_t:=\Phi(t,\cdot)$ is a one-parameter group of $\C^1$-diffeomorphisms of $\R^N$ into itself satisfying $\Phi_0={\rm id}$. Then the requirement $X\cdot\nu_\O=0$ on $\partial\O$ implies that $\Phi_t(\partial\Omega)=\partial\O$ for every 
$t\in\R$. Hence (the restriction of) $\Phi_t$ is a $\C^1$-diffeomorphism of $\partial\O$ into itself, and a $\C^1$-diffeomorphism of $\O$ into itself. 

\begin{definition}\label{def:inner_variations_MS}
Let $u \in SBV^2(\O)$, $X\in\C_c^1(\R^N;\R^N)$, {and $G \in H^1(\R^N)$ satisfying $X\cdot\nu_\O=0$ and $G=g$} on $\partial\O$. Setting $\{\Phi_t\}_{t\in\R}$ to be the integral flow of $X$ and 
$$u_t:= u\circ \Phi_t^{-1}-G\circ \Phi_t^{-1} + G \in SBV^2(\O)\,,$$ 
the first and second inner variations of $MS$ at $u$ are defined by
$$\delta MS(u){[X,G]}:=\frac{\dd}{\dd t}\Bigl|_{t=0} MS(u_t)\,, \quad \delta^2 MS(u){[X,G]}:=\frac{\dd^2}{\dd t^2}\Bigl|_{t=0} MS(u_t)\,.$$
\end{definition}
 It can be checked that, provided $\partial\O$, $g$, and $G$ are smooth enough, the above derivatives exist and they can be explicitly computed (see Lemma \ref{lem:explicit_inner_stab2}). Analogously, we define inner variations of the Ambrosio-Tortorelli functional.

\begin{definition}\label{def:inner_variations_AT} 
\rm{Let $(u,v) \in \mathcal A_g(\O)$, {$X\in\C_c^1(\R^N;\R^N)$, and $G \in H^1(\R^N)$ satisfying $X\cdot\nu_\O=0$ and $G=g$} on $\partial\O$. We set 
$$(u_t,v_t):=\big(u\circ \Phi_t^{-1}-G\circ \Phi_t^{-1}+G,v \circ \Phi_t^{-1}\big) \in \mathcal A_g(\O)\,.$$
 We  define the first and second inner variations of $AT_\e$ at  \( (u,v) \) by 
\begin{equation}\label{eq:def_delta2}
\delta AT_\e(u,v){[X,G]}:=\frac{\dd}{\dd t}\Bigl|_{t=0} AT_\e (u_t,v_t) ,\quad \delta^2 AT_\e(u,v){[X,G]}:=\frac{\dd^2}{\dd t^2}\Bigl|_{t=0} AT_\e(u_t,v_t)\,.
\end{equation}
}
\end{definition}
Once again, the limits in \eqref{eq:def_delta2} exist {whenever $\partial\O$, $g$, and $G$ are sufficiently smooth,} and one can compute them explicitly (see Lemmas \ref{lem:lien_inner_outer} \& \ref{lem:expressions_inner}).
\vskip5pt

We emphasize that we are considering in Definitions \ref{def:inner_variations_MS} \& \ref{def:inner_variations_AT}  deformations {\sl up to the boundary}. Compare to the usual deformations involving compactly supported  perturbations in $\O$ of the original maps, it requires the additional test function $G$. This is of fundamental importance for the $MS$ functional to recover information at the boundary since the Dirichlet boundary condition is implemented in the functional as a penalization. Of course, the type of deformations we are using includes as a particular case the usual ones defined only through a vector field $X$ compactly supported in $\O$, see Remark~\ref{remgendeformations}.

\subsection{First order criticality  conditions for $MS$}

In view of the discussion above, the non smooth character of $MS$ forces the appropriate notion of critical point to involve both outer and inner variations. In other words, a critical point of the Mumford-Shah functional is a critical point with respect to both outer and inner variations, a property obviously satisfied by global (and even local) minimizers.

\begin{definition}
{\rm Let $\O \subset \R^N$ be a bounded open set {with boundary of class at least $\mathscr{C}^2$} and {$g \in \mathscr{C}^2(\partial\O)$}. A function $u_*\in SBV^2(\O)$ is a critical point of the Mumford-Shah functional if 
\begin{equation}\label{cond1MS}
\dd MS(u_*)[\phi]=0\quad \text{ for all $\phi \in SBV^2(\O)$ with $\widehat J_\phi \subset \widehat J_{u_*}$}\,,
\end{equation}
and 
\begin{equation}\label{cond2MS}
\delta MS(u_*)[X,G]=0
\end{equation}
for all $X\in\C_c^1(\R^N;\R^N)$ and $G\in \mathscr{C}^2(\R^N)$ satisfying $X\cdot\nu_\O=0$ and $G=g$ on $\partial\O$.}
\end{definition}

From these criticality conditions, one can derive a set of Euler-Lagrange equations which can be written in a strong form if the smoothness of $u_*$ and \(\widehat{J}_{u_*}\) allows it.  
Specializing first the condition \eqref{cond1MS} to \(\phi \in \C^\infty_c(\O)\) yields
\begin{equation}\label{eq:critic_MS1}
\dive (\nabla u_*)=0\quad \text{ in } \mathscr{D}^\prime(\O)\,.
\end{equation}
Then, if  \(\widehat{J}_{u_*}\) is regular enough,  one can choose test functions $\phi$ in  \eqref{cond1MS} with a non trivial jump set but smooth up to  \(\widehat{J}_{u_*}\) from both sides. It leads to the following homogeneous Neumann condition 
\begin{equation}\label{eq:critic_MS2_Neumann}
\p_\nu u_*=0 \quad\text{ on } \widehat{J}_{u_*}\,,
\end{equation}
see \cite[formula (7.42)]{Ambrosio_Fusco_Pallara_2000}. In other words, allowing test functions $\phi$ in  \eqref{cond1MS}  with $\widehat J_\phi \subset \widehat J_{u_*}$ (and not only in  \(\phi \in \C^\infty_c(\O)\)) provides the weak formulation of \eqref{eq:critic_MS2_Neumann} which complements \eqref{eq:critic_MS1}. 

Computing $\delta MS(u_*)[X,G]$ (see formula \eqref{eq:1146}) and using equation \eqref{eq:critic_MS1}, the stationarity condition \eqref{cond2MS} appears to be independent of the test function $G$ and it reduces to
\begin{multline}\label{def:ctMS}
\int_\O \left(|\nabla u_*|^2 {\rm Id}-2\nabla u_* \otimes \nabla u_*\right):D X\dd x+\int_{\widehat J_{u_*}}  \dive^{\widehat J_{u_*}} X\, \dd\HH^{N-1}=-2\int_{\partial\O} (\nabla u_*\cdot\nu_\O) (X\cdot \nabla_\tau g)\,\dd\mathcal{H}^{N-1}  \\
 \text{for all } X \in \C^1_c(\R^N;\R^N) \text{ with } X\cdot \nu_\O =0 \text{ on } \p \O\,.
\end{multline} 
Here \( \dive^{\widehat J_{u_*}}X =\tr \left( ({\rm Id}-\nu_{u_*}\otimes \nu_{u_*}) D X \right)\) is the tangential divergence of $X$ on the countably $\HH^{N-1}$-rectifiable set $\widehat J_{u_*}$ with \(\nu_{u_*}\) the approximate unit normal to that set. The boundary term in the right hand side of \eqref{def:ctMS} is interpreted in the sense of duality by \eqref{eq:critic_MS1}, and $\nabla_\tau g$ denotes the tangential derivative of $g$. If $J_{u_*}$ and $u_*$ are regular enough, then \eqref{def:ctMS} provides the coupling equation 
$$H_{u_*}+\big[|\nabla u_*|^2\big]^\pm=0 \quad\text{ on } J_{u_*}\,,$$
where $H_{u_*}$ denotes the scalar mean curvature of $J_{u_*}$ with respect to the normal  \(\nu_{u_*}\) and $\big[|\nabla u_*|^2\big]^\pm$ the (accordingly oriented) jump of $|\nabla u_*|^2$ across $J_{u_*}$ (see \cite[Chapter 7, Section 7.4]{Ambrosio_Fusco_Pallara_2000}).

\begin{remark}[\bf 1D case]\label{rem1D} 
{\rm In the one-dimensional case $N=1$, if \( \O= (0,L)$ for some $L>0$, we can see that if \(u  \in SBV^2(0,L)\) satisfies conditions \eqref{eq:critic_MS1}-\eqref{eq:critic_MS2_Neumann}, then \(u \) is either piecewise constant with a finite number of jumps or \(u\) is a globally affine function (with no jump). Indeed, the very definition of $SBV^2(0,L)$ shows that \(u\) has a finite number  of jumps. Then, condition \eqref{eq:critic_MS1} implies that $u$ is affine in between to consecutive jump points, and \eqref{eq:critic_MS2_Neumann} implies that the slope of all affine functions must be zero. However, condition \eqref{def:ctMS} does not play any role because it only implies that \( |u'|\) is constant in \( (0,L)\), where $u'$ is the approximate derivative of $u$. From this, we just deduce that $u$ is a piecewise affine function  with equal slopes in absolute value, and it is not sufficient by itself to prove that \(u\) is piecewise constant. It indicates that the use of $SBV^2$-test functions in \eqref{cond1MS} can not be relaxed to a class of smooth  functions (in any dimension).} 
\end{remark}

\subsection{Main results}
As already mentioned, the main purpose of this article is  to investigate the asymptotic behavior of critical points of the Ambrosio-Tortorelli functional as $\e\to 0$. In view of the $\Gamma$-convergence result, one may expect that critical points converge to critical points, possibly under the assumption of convergence of energies. Without fully resolving this question, our analysis provides the first  answer in this direction in arbitrary dimensions showing that a limit of critical points of $AT_\e$ must  at least be a critical point of $MS$ with respect to inner variations, i.e., a stationary point of $MS$. If a critical point $(u_\e,v_\e)$ of  $AT_\e$ is smooth enough, then it is easy to see that it is also stationary, i.e., $\delta AT_\e(u_\e,v_\e)=0$ (see Lemma \ref{lem:lien_inner_outer}). Hence, if regularity of critical points $AT_\e$ holds, proving the convergence of the first inner  variations implies the announced stationarity of the limit. This is the path we have followed, and the regularity issue is the object of our first main theorem.

\begin{theorem}\label{thmmain1}
Let $\Omega\subset \mathbf{R}^N$ be a bounded open set with Lipschitz boundary and $g\in H^{\frac12}(\partial\Omega)$. If $(u_\e,v_\e)\in\mathcal A_g(\Omega)$ is a critical point of $AT_\e$, then $(u_\e,v_\e)\in [\C^\infty(\Omega)]^2$ and the following  regularity up to the boundary holds. 
\begin{itemize}
\item[\it (i)] If $g\in H^{\frac12}(\partial\Omega)\cap L^\infty(\p\O)$, then $u_\e\in L^\infty(\O)$. 
\vskip5pt
\item[\it (ii)] If $\partial\O$ is of class $\C^{k\vee 2,1}$ and $g\in \C^{k,\alpha}(\partial\O)$ with $k\geq 1$ and $\alpha\in(0,1)$, then $(u_\e,v_\e)\in \C^{k,\alpha}(\overline\O)\times \C^{k\vee2,\alpha}(\overline\O)$. 
\end{itemize}
\end{theorem}

We emphasize that the regularity in Theorem \ref{thmmain1} is highly non trivial since the second equation in \eqref{eq:PDE} is of the form $\Delta v=f$ with $f\in L^1$ and standard  linear elliptic theory does not directly apply. Instead, we shall rely on arguments borrowed from the regularity theory for harmonic maps into a manifold, or more generally for variational nonlinear elliptic systems, see e.g.\ \cite{Giaquinta_Martinazzi_2012}. The key issue is to prove H\"older continuity of $v_\e$, that we achieve by proving that it belongs to a suitable Morrey-Campanato space. We treat interior and boundary regularity in a similar way through a reflection argument of independent interest originally devised in \cite{Scheven_2006}.

\medskip

In our second main theorem, we show that, under the assumption of convergence of energies, limits (up to a subsequence) of critical points of \(AT_\e\) are critical points of \(MS\) for the inner variations.

\begin{theorem}\label{thm:main} 
Assume that $\O \subset \R^N$ is a bounded open set of class $\C^{2,1}$ and $g \in \C^{2,\alpha}(\partial\O)$ for some $\alpha \in (0,1)$. Let $\{(u_\e,v_\e)\}_{\e>0}\subset \mathcal A_g(\O)$ be a family of critical points of the Ambrosio-Tortorelli functional. Then, the following properties hold:
\begin{itemize}
\item[\it (i)] If the energy bound
\begin{equation}\label{eq:hypnrj}
\sup_{\e>0}AT_\e(u_\e,v_\e)<\infty
\end{equation}
is satisfied, up to a subsequence, $u_\e \to u_*$ strongly in $L^2(\O)$ as $\e\to 0$ for some $u_* \in SBV^2(\O) \cap L^\infty(\O)$ satisfying $\nabla u_*\cdot\nu_\O \in L^2(\partial\O)$, and $\dd MS(u_*)[\phi]=0$ for all \(\phi \in \C^\infty_c(\O)\), i.e.,  
$$\dive(\nabla u_*)=0 \quad \text{ in }\mathscr{D}^\prime(\O)\,.$$
\item[\it (ii)] If, further, the energy convergence
\begin{equation}\label{eq:hypothesis}
AT_\e(u_\e,v_\e) \to MS(u_*)
\end{equation}
is satisfied, then $\delta MS(u_*)=0$, i.e.,  
\begin{equation}\label{eq:first-var3}
\int_\O \left( |\nabla u_*|^2 {\rm Id} -2\nabla u_* \otimes \nabla u_*\right):D X\dd x+\int_{\widehat J_{u_*}}  \dive^{\widehat J_{u_*}} X\, \dd\HH^{N-1}=-2\int_{\partial\O} (\nabla u_*\cdot \nu_\O) (X\cdot\nabla_\tau g)\dd\HH^{N-1}
\end{equation}
 for all vector field $X \in \C^1_c(\R^N;\R^N)$ with $X\cdot\nu_\O=0$ on $\partial\O$. 
\end{itemize}
\end{theorem}

\begin{remark}{\rm 
At this stage, it is still open whether or not $u_*$ is a critical point of $MS$ as we do not know if the outer variation $\dd MS(u_*)$ also vanishes on arbitrary functions $\phi\in SBV^2(\O)$ satisfying $\widehat J_\phi \subset \widehat J_{u_*}$ (and not only on $ \C^\infty_c(\O)$). In other words,  the weak form of the homogeneous Neumann condition \eqref{eq:critic_MS2_Neumann} on $\widehat J_{u_*}$ remains to be established. This is the only missing ingredient to obtain that $u_*$ is a critical point of $MS$. 
}
\end{remark}

 An assumption of convergence of energies similar to \eqref{eq:hypothesis} has been used in \cite{Luckhaus_Modica_1989,Le_2011,Le_2015,Le_Sternberg_2019} to prove that critical points of the Allen-Cahn functional (from phase transitions) converge towards critical points of the perimeter functional, hence to minimal surfaces. The analysis without this assumption  has been first carried out in  \cite{Hutchinson_Tonegawa_2000}, and it shows that critical points converge (in the sense of inner variations) towards integer multiplicity stationary varifolds, a measure theoretic generalization of minimal surfaces allowing for multiplicities. Interfaces with multiplicities do appear as limits of critical points of the Allen-Cahn energy and cannot be excluded, see e.g.\ \cite[Section 6.3]{Hutchinson_Tonegawa_2000}. In our context, a similar phenomenon may appear, so that assumption \eqref{eq:hypothesis} is probably necessary. 

In   \cite{Le_2011,Le_2015,Le_Sternberg_2019},  convergence of  energies is also used to pass to the limit in the second inner variation. Following the same path, 
 \eqref{eq:hypothesis}  allows us to pass to the limit in the second inner variation of \(AT_\e\).  
 It shows that the second inner variations of $AT_\e$ {\sl do not} converge to the second inner variation of $MS$, but to the second inner variation plus a residual additional term. As a byproduct, it follows that limits of  stable critical points of $AT_\e$  
 satisfy an ``augmented'' second order minimality condition. Second order minimality criteria for $MS$ has been addressed in \cite{Cagnetti_Mora_Morini_2008,Bonacini_Morini_2015}. We also note that the convergence of the second inner variation for the Allen-Cahn functional without the assumption of convergence of energies has been studied in \cite{Gaspar_2020}, see also \cite{Hiesmayr_2018}. Convergence of second inner variations is  our third and last main result.

\begin{theorem}\label{thm:main_2}
Assume that $\O\subset \R^N$ is a bounded open set of class $\C^{3,1}$ and $g \in \C^{3,\alpha}(\partial\O)$ for some $\alpha \in (0,1)$. Let $\{(u_\e,v_\e)\}_{\e>0} \subset \mathcal A_g(\O)$ be a family of critical points of the Ambrosio-Tortorelli functional and $u_* \in SBV^2(\O) \cap L^\infty(\O)$ be as in Theorem \ref{thm:main}, satisfying the convergence of energy \eqref{eq:hypothesis}. Then,  
\begin{itemize}
\item[\it (i)]For all  \( X \in \C^2_c(\R^N;\R^N)\) and all $G \in \mathscr C^3(\R^N)$ with $X\cdot\nu_\O=0$ and $G=g$ on $\partial\O$,
\begin{equation}\label{eq:main2_limit_inner}
\lim_{\e \to 0} \delta^2 AT_\e (u_\e,v_\e)[X,G]= \delta^2 MS(u_*)[X,G]+\int_{\widehat J_{u_*}}|D X :(\nu_{u_*}\otimes \nu_{u_*})|^2 \dd \mathcal{H}^{N-1}\,;
\end{equation}
\item[\it (ii)] If $(u_\e,v_\e)$ is a stable critical point of $AT_\e$, i.e.\, 
$$\dd^2 AT_\e(u_\e,v_\e)[\phi,\psi] \geq 0 \quad \text{ for all }(\phi,\psi) \in [\C_c^\infty(\O)]^2,$$
then $u_*$ satisfies the second order inequality
\begin{equation}\label{secordmincritthm}
\delta^2 MS(u_*)[X,G]+\int_{\widehat J_{u_*}}|D X :(\nu_{u_*} \otimes \nu_{u_*})|^2 \dd \mathcal{H}^{N-1}\geq 0
\end{equation}
for all \( X \in \C^2_c(\R^N;\R^N)\) and all $G \in \mathscr C^3(\R^N)$ with $X\cdot\nu_\O=0$ and $G=g$ on $\partial\O$.
\end{itemize}
\end{theorem}

\medskip

In the one-dimensional case,  the asymptotic analysis as $\e\to 0$ of critical points of the Ambrosio-Tortorelli functional has already been carried out in \cite{Francfort_Le_Serfaty_2009,Le_2010} for different sets of boundary conditions. In \cite{Francfort_Le_Serfaty_2009}, a homogeneous Neumann boundary condition is assumed for the phase field variable $v$. The authors proved that if \( \{(u_\e,v_\e)\}_{\e>0}\) is a family of critical points of the Ambrosio-Tortorelli functional satisfying \eqref{eq:hypnrj}, then, up to a subsequence, \((u_\e,v_\e)\to (u,1)\) in \([L^2(\O)]^2\) with \(u \in SBV^2(\O)\) that is either globally affine or piecewise constant with a finite number of jumps, see Remark \ref{rem1D}. This result is extended in  \cite{Le_2010} to the Ambrosio-Tortorelli functional with a fidelity term. Note that our present analysis also applies in the presence of a fidelity term, but we do not consider this case here in order not to add useless difficulties. In a short note \cite{BMR1D}, we have also carried out the 1D analysis in our setting, i.e., with the Dirichlet boundary condition on the $v$ variable. In this case, we have established a convergence result for critical points without assuming the convergence of the energy \eqref{eq:hypothesis}, but proving \eqref{eq:hypothesis} as a consequence of the energy bound  \eqref{eq:hypnrj}. 
 It allows us to exhibit non-minimizing critical points of $AT_\e$ satisfying our energy convergence assumption~\eqref{eq:hypothesis} (see \cite[Remark 1.2]{BMR1D}).

\subsection{Ideas of the convergence proof}

The proof of Theorem \ref{thm:main} relies on the classical compactness argument and the lower bound inequality for the Ambrosio-Tortorelli functional. Indeed, the energy bound for a family \( \{(u_\e,v_\e)\}_{\e>0}\subset \A_g(\O)\) of critical points for $AT_\e$ implies the $L^2(\O)$-convergence (up to a subsequence) of \(u_\e\) to a limit \(u_* \in SBV^2(\O)\), together with a \(\Gamma\)-liminf inequality $MS(u_*) \leq \liminf_\e AT_\e(u_\e,v_\e)$. Our energy convergence assumption \eqref{eq:hypothesis} leads to the {\sl equipartition} of phase field energy, as well as the convergence of the bulk energy. Then, as in \cite{Hutchinson_Tonegawa_2000}, we associate an \((N-1)\)-varifold \(V_\e\) to the phase field variable \(v_\e\), which converges (again up to a subsequence) to a limiting varifold~\(V_*\). The energy convergence \eqref{eq:hypothesis} allows us to identify the mass of \(V_*\), that is \( \|V_*\|=\mathcal{H}^{N-1}\res{\widehat{J}_{u_*}}\). Next, we  use the equations satisfied by \( (u_\e,v_\e)\) in their conservative form to pass to the limit, and find an equation satisfied by \(u_*\) and \(V_*\). The idea is then to employ a blow-up argument similar to \cite{Ambrosio_Soner_1997} to identify (the first moment of) \(V_*\), and show that it is the rectifiable varifold associated to \(\widehat{J}_{u_*}\) with multiplicity one. 

To prove Theorem \ref{thm:main_2}, we argue as in \cite{Le_2011,Le_2015,Le_Sternberg_2019}. We observe that the convergence \(V_\e \rightharpoonup V_*\) in the sense of varifolds and the identification of \(V_*\) implies the convergence of quadratic terms \(\e \nabla v_\e \otimes \nabla v_\e \rightharpoonup \frac12 \nu_{u_*}\otimes \nu_{u_*} \mathcal{H}^{N-1}\res{\widehat{J}_{u_*}}\) in the sense of measures. This information is precisely what is needed to pass to the limit in the second inner variation of \(AT_\e\), and we infer from a stability condition on \( (u_\e,v_\e) \in \A_g(\O)\) a stability condition on the limit \(u_* \in SBV^2(\O)\).

\medskip

The paper is organised as follows. Section \ref{sec:2} collects several notation that will be used throughout the paper. In Section \ref{sec:regularity}, we study the regularity theory for critical points of the Ambrosio-Tortorelli functional proving first smoothness in the interior of the domain, and then  smoothness at the boundary.  In Section \ref{sec:compactness}, we prove compactness of a family \( \{(u_\e,v_\e)\}_{\e>0}\) satisfying a uniform energy bound $\sup_\e AT_\e(u_\e,v_\e)<\infty$. The regularity result allows one to derive the conservative form of the equations satisfied by these critical points which itself provides bounds on the normal traces of $u_\e$ and $v_\e$ on $\partial\O$. Then, in Section \ref{sec:convergence}, we improve the previous results by assuming the energy convergence \( AT_\e(u_\e,v_\e) \to MS(u_*)\). From this assumption we obtain equipartition of  the phase field part of the energy. Then, we employ a reformulation in terms of varifolds to pass to the limit in the inner variational equations satisfied by critical points of $AT_\e$ to prove that the weak limit \(u_*\) of \( u_\e\) is a stationary point of the Mumford-Shah energy. 
The asymptotic behavior of the second inner variations is performed  in Section \ref{sec:passing_limit_2nd_variation}. 

\section{Notation and preliminaries}\label{sec:2}

\subsection{Measures} 

The Lebesgue measure in $\R^N$ is denoted by $\LL^N$, and the $k$-dimensional Hausdorff measure by $\HH^k$. We will sometime write $\omega_k$ for the $\LL^k$-measure of the $k$-dimensional unit ball in $\R^k$.

 If $X \subset \R^N$ is a locally compact set and $Y$ an Euclidean space, we denote by $\mathcal M(X;Y)$ the space of $Y$-valued bounded Radon measures in $X$ endowed with the norm $\| \mu \|=|\mu|(X)$, where $|\mu|$ is the variation of the measure $\mu$. If $Y=\R$, we simply write $\mathcal M(X)$ instead of $\mathcal M(X;\R)$. By Riesz representation theorem, $\mathcal M(X;Y)$ can be identified with the topological dual of $\C_0(X;Y)$, the space of continuous functions $f:X \to Y$ such that $\{|f|\geq \e\}$ is compact for all $\e>0$. The weak* topology of $\mathcal M(X;Y)$ is defined using this duality.

\subsection{Functional spaces}

We use standard notation for Lebesgue, Sobolev and H\"older spaces. Given a bounded open set $\O \subset \R^N$, the space of functions of bounded variation is defined by
$$BV(\O)=\{u \in L^1(\O) : \; Du \in \mathcal M(\O;\R^N)\}.$$
We shall also consider the subspace $SBV(\O)$ of special functions of bounded variation made of functions $u \in BV(\O)$ whose distributional derivative can be decomposed as
$Du=\nabla u \LL^N + (u^+-u^-)\nu_u \HH^{N-1} \res J_u$. 
In the previous expression, $\nabla u$ is the Radon-Nikod\'ym derivative of $Du$ with respect to $\LL^N$, and it is called the approximate gradient of $u$. The Borel set $J_u$ is the (approximate) jump set of $u$. It is a countably $\HH^{N-1}$-rectifiable subset of $\O$ oriented by the (approximate) normal direction of jump $\nu_u :J_u \to \mathbf S^{N-1}$, and $u^\pm$ are the one-sided approximate limits of $u$ on $J_u$ according to $\nu_u$. Finally we define
$$SBV^2(\O)=\{u \in SBV(\O) : \; \nabla u \in L^2(\O;\R^N) \text{ and } \HH^{N-1}(J_u)<\infty\}.$$

\subsection{Varifolds}\label{sec:varifold}

Let us recall several basic ingredients of the theory of varifolds (see \cite{Simon_1983} for a detailed description). We denote by $\mathbf G_{N-1}$ the Grassmannian manifold of all $(N-1)$-dimensional linear subspaces of $\R^N$. The set $\mathbf G_{N-1}$ is as usual identified with the set of all orthogonal projection matrices onto $(N-1)$-dimensional linear subspaces of $\R^N$, i.e., $N \times N$ symmetric matrices $A$ such that $A^2=A$ and ${\rm tr}(A)=N-1$, in other words, matrices of the form $A={\rm Id}-e\otimes e$
for some $e \in \mathbf S^{N-1}$.

A $(N-1)$-varifold in $X$ (a locally compact subset of $\R^N$) is a bounded Radon measure on $X\times \mathbf G_{N-1}$. The class of $(N-1)$-varifold in $X$ is denoted by $\mathbf V_{N-1}(X)$. The mass of $V\in \mathbf V_{N-1}(X)$ is simply the measure  $\|V\| \in \M(X)$ defined by $\|V\|(B)=V(B \times \mathbf G_{N-1})$ for all Borel sets $B \subset X$.  We define the first variation of an \((N-1)\)-varifold in $V$ in an open set \(U \subset \R^N\) by
$$\delta V(\varphi)=\int_{U \times \mathbf{G}_{N-1}} D \varphi(x):A \dd V(x,A) \quad \text{ for all } \varphi\in \C^1_c(U;\R^N)\,.$$
We say that an \((N-1)\)-varifold is stationary in $U$ if \(\delta V(\varphi)=0\) for all \(\varphi\) in \(\C^1_c(U;\R^N)\). We recall that such a varifold satisfies the monotonicity formula
$$\frac{\|V\|(B_\varrho(x_0))}{\varrho^{N-1}}=\frac{\|V\|(B_r(x_0))}{r^{N-1}}+\int_{(B_\varrho(x_0)\setminus B_r(x_0)) \times \mathbf G_{N-1}}\frac{|P_{A^\perp}(x-x_0)|^2}{|x-x_0|^{N+1}}\dd V(x,A)$$ 
for all $x_0 \in U$ and $0<r<\varrho$ with $B_\varrho(x_0) \subset U$, where $P_{A^\perp}$ is the orthogonal projection onto the one-dimensional space $A^\perp$ (see \cite[paragraph 40]{Simon_1983}).

\subsection{Tangential divergence}

Let $\Gamma$ be a countably $\HH^{N-1}$-rectifiable set and let $T_x \Gamma$ its approximate tangent space defined for $\HH^{N-1}$-a.e.\ $x \in \Gamma$. We consider an orthonormal basis $\{\tau_1(x),\ldots,\tau_{N-1}(x)\}$ of $T_x \Gamma$ and denote by $\nu(x)$ a normal vector to $T_x\Gamma$. If $\zeta:\R^N \to \R^N$ is a smooth vector field, we denote by
$$\dive^\Gamma \zeta:=\sum_{i=1}^{N-1} \tau_i \cdot \partial_{\tau_i} \zeta=({\rm Id}-\nu \otimes \nu) : D \zeta$$
the tangential divergence, and  $(\partial_{\tau_i} \zeta)^\perp=((\partial_{\tau_i}\zeta)\cdot \nu) \nu=\partial_{\tau_i}\zeta- \sum_{j=1}^{N-1} (\tau_j \cdot \partial_{\tau_i} \zeta)\tau_j$.

\section{Regularity theory for critical points of the Ambrosio-Tortorelli energy}\label{sec:regularity}

In this section, we investigate interior and boundary regularity properties of critical points of the Ambrosio-Tortorelli functional $AT_\e$ for a parameter $\e>0$ which is kept fixed. 

\subsection{Interior regularity}

We first establish interior regularity following ideas used by T. Rivi\`ere in \cite{Riviere_thesis_1993} to prove the regularity of harmonic maps with values into a revolution torus. 

\begin{theorem}\label{prop:int-regularity} 
Let $\O\subset\R^N$ be a bounded open set. If $(u_\e,v_\e)\in H^1(\O)\times [H^1(\O)\cap L^\infty(\O)]$ satisfies \eqref{eq:varformu}-\eqref{eq:varformv}, then
$(u_\e,v_\e)\in [\C^\infty(\O)]^2$. 
\end{theorem}

\begin{proof}
For simplicity, we drop the subscript \(\e\) in \((u_\e,v_\e)\) and write instead $(u,v)$. We also assume \(N \geq 2\) since in the case \(N=1\), the regularity of \( (u,v)\) solution of \eqref{eq:varformu}-\eqref{eq:varformv} is elementary. 

By  \eqref{eq:varformu}, $u$ weakly solves 
\begin{equation}\label{weakeq1}
-{\rm div}\big((\eta_\e+v^2)\nabla u\big)=0\quad\text{in $\O$}\,. 
\end{equation}
Setting $M:=\|v\|_{L^\infty(\O)}$, the matrix field \( (\eta_\e+v^2)\text{Id}\) has bounded measurable coefficients and it satisfies $\eta_\e\text{Id}\leq (\eta_\e+v^2)\text{Id}\leq (\eta_\e+M^2)\text{Id}$ a.e.\ in $\O$ in the sense of quadratic forms. 
It is therefore uniformly elliptic and  the De Giorgi-Nash-Moser regularity theorem applies to  equation \eqref{weakeq1}.  It provides the existence of \(\alpha\in (0,1)\) such that \(u \in \C^{0,\alpha}_{\text{loc}}(\O)\) together with the estimate: 
\begin{equation}\label{eq:estimate_Morrey}
K(\omega):=\sup_{ x_0 \in \omega, \, \varrho>0, \, {B}_\varrho(x_0) \subset\omega} \frac{1}{\varrho^{N-2+2\alpha}} \int_{B_\varrho(x_0)} |\nabla u|^2 \dd x<\infty
\end{equation}
for every open subset ${\omega}$ such that  $\overline\omega\subset \O$ (see e.g.\ Theorem 8.13 and Eq. (8.18) in \cite{Giaquinta_Martinazzi_2012}).
\vskip3pt

Now we claim that the function \(v\) belongs to \(\C^{0,\alpha}_{\text{loc}}(\O)\). Before proving this claim, we complete the proof of the theorem. Assuming the claim to be true, we can use the Schauder estimates (see e.g.\ \cite[Theorem 5.19]{Giaquinta_Martinazzi_2012}) to derive from equation \eqref{weakeq1} that \(u\in \C^{1,\alpha}_{\text{loc}}(\O)\). On the other hand, by \eqref{eq:varformv}, $v$ weakly solves 
\begin{equation}\label{weakeq2}
-\e\Delta v=\frac{1-v}{4\e}-|\nabla u|^2v \quad\text{in $\O$}\,. 
\end{equation}
Since the right-hand-side of \eqref{weakeq2} belongs to $\C^{0,\alpha}_{\rm loc}(\O)$, it follows from standard Schauder estimates that $v\in \C^{2,\alpha}_{\rm loc}(\O)$. By a classical bootstrap, it now follows from equations \eqref{weakeq1} and  \eqref{weakeq2} that both $u$ and $v$ are of class $\C^\infty$ in~$\O$. 

\vskip3pt

Hence, it only remains to show the claim \(v\in \C^{0,\alpha}_{\text{loc}}(\O)\). To this purpose,  we fix an arbitrary ball $\overline{B}_{2R}(x_0)\subset\Omega$, and 
we aim to prove that  \(v\in \C^{0,\alpha}_{\text{loc}}(B_R(x_0))\). Consider $v_1 \in H^1(B_{2R}(x_0))$ to be the unique weak solution of 
\begin{equation}\label{eq:equation_v1}
\left\{
\begin{array}{rcll}
 \displaystyle -\Delta v_1&=&  \displaystyle  \frac{1-v}{4\e^2} & \text{ in } B_{2R}(x_0)\,, \\[8pt]
v_1 &=& v & \text{ on } \p B_{2R}(x_0)\,.
\end{array}
\right.
\end{equation}
Since \( \Delta v_1 \in L^\infty(B_{2R}(x_0))\),  the Calder\'on-Zygmund estimates yield \(v_1 \in W^{2,p}_{\text{loc}}(B_{2R}(x_0))\) for every \(p<\infty\). By Sobolev embedding, it follows that \(v_1 \in C^{1,\beta}_{\text{loc}}(B_{2R}(x_0))\) for every \(\beta\in(0,1)\). In particular, we have $v_1\in L^\infty(B_R(x_0))$. 
 
Set $v_2:=v-v_1 \in H_0^1(B_{2R}(x_0))$. By \eqref{weakeq2} and \eqref{eq:equation_v1}, $v_2$ is a weak solution of
\begin{equation}\label{eq:equation_v2}
 -\Delta v_2= -\frac{1}{\e}|\nabla u|^2v  \quad \text{ in } B_{2R}(x_0)\,.
\end{equation}
To show that $v_2\in \C^{0,\alpha}_{\rm loc}(B_R(x_0))$, the Morrey-Campanato Theorem (see e.g. \cite[Theorem 5.7]{Giaquinta_Martinazzi_2012}) ensures that it suffices to prove the following Morrey type estimate: 
\begin{equation}\label{eq:Morrey}
\sup_{y \in B_R(x_0),\, \varrho\in(0,R)} \,\frac{1}{\varrho^{N-2+2\alpha}} \int_{B_\varrho(y)} |\nabla v_2|^2\dd x<\infty\,.
\end{equation}
Let $y\in B_R(x_0)$ and $r\in(0,R)$ arbitrary. We denote by \(w \in v_2+H^1_0(B_r(y))\) the harmonic extension of \(v_2\) in the ball \(B_r(y)\), i.e., the unique (weak) solution of
\begin{equation}\label{eq:equation_w}
\left\{
\begin{array}{rcll}
-\Delta w &=& 0 & \text{ in } B_r(y)\,, \\
w&=& v_2 & \text{ on } \p B_r(y)\,.
\end{array}
\right.
\end{equation}
Since $v_2=v-v_1\in L^\infty(B_R(x_0))$, we have $|w|\leq \|v_2\|_{L^\infty(B_R(x_0))}$ on $\p B_r(y)$, and the weak maximum principle implies that $w\in L^\infty(B_r(y))$ with $\|w\|_{L^\infty(B_r(y))}\leq \|v_2\|_{L^\infty(B_R(x_0))}$. Moreover, $|\nabla w|^2$ being subharmonic in $B_r(y)$, we get that for every $\varrho<r$,
$$\int_{B_\varrho(y)} |\nabla w|^2\dd x \leq \left( \frac{\varrho}{r} \right)^N\int_{B_r(y)} |\nabla w|^2\dd x\,.$$
Recalling $w$ also minimizes the Dirichlet integral among all functions agreeing with $v_2$ on $\p B_r(y)$, we infer that 
\begin{align*}
\nonumber\int_{B_\varrho(y)} |\nabla v_2|^2\dd x & \leq  2\int_{B_\varrho(y)} |\nabla w|^2\dd x+2\int_{B_\varrho(y)} \big|\nabla (w-v_2)\big|^2\dd x \\
\nonumber & \leq 2 \left( \frac{\varrho}{r}\right)^N \int_{B_r(y)} |\nabla w|^2\dd x+2\int_{B_r(y)} \big |\nabla (w-v_2)\big|^2\dd x \\
& \leq  2 \left(  \frac{\varrho}{r}\right)^N \int_{B_r(y)} |\nabla v_2|^2\dd x+2 \int_{B_r(y)}\big|\nabla(w-v_2)\big|^2\dd x
\end{align*}
for every \(\varrho < r\).  Since $w-v_2=0$ on $\p B_r(y)$, \eqref{eq:equation_v2} and \eqref{eq:equation_w} lead to 
$$  \int_{B_r(y)}\big|\nabla(w-v_2)\big|^2\dd x=\frac{1}{\e} \int_{B_r(y)}|\nabla u|^2v(w-v_2)\dd x\leq \frac{2}{\e}\|v\|_{L^\infty(\O)}\|v_2\|_{L^\infty(B_R(x_0))} \int_{B_r(y)}|\nabla u|^2\dd x\,.$$
In view of \eqref{eq:estimate_Morrey}, we have thus proved that for every $y\in B_R(x_0)$ and $0<\varrho\leq r<R$, 
$$\int_{B_\varrho(y)} |\nabla v_2|^2\dd x  \leq  2 \left(  \frac{\varrho}{r}\right)^N \int_{B_r(y)} |\nabla v_2|^2\dd x+ C_1 r^{N-2+2\alpha} $$
with $C_1:=\frac{4}{\e}\|v\|_{L^\infty(\O)}\|v_2\|_{L^\infty(B_R(x_0))}K(B_{2R}(x_0))$. By using a classical iteration lemma (see e.g.\ \cite[Lemma~5.13]{Giaquinta_Martinazzi_2012}), we infer that for every $y\in B_R(x_0)$ and $0<\varrho<R$,
$$\int_{B_\varrho(y)} |\nabla v_2|^2\dd x  \leq  C_\alpha \varrho^{N-2+2\alpha}\bigg(\frac{1}{R^{N-2+2\alpha}}\int_{B_{2R}(x_0)} |\nabla v_2|^2\dd x +C_1\bigg)\,,$$
for a constant $C_\alpha$ depending only on $\alpha$ and $N$. Hence $v_2$ satisfies the Morrey estimate \eqref{eq:Morrey}, and thus \(v_2\in \C^{0,\alpha}_{\text{loc}}(B_R(x_0))\). In turn, \(v=v_1+v_2\in \C^{0,\alpha}_{\text{loc}}(B_R(x_0))\) and the proof of the claim is complete.
\end{proof}

\subsection{Maximum principle and boundary regularity}

We first show a (standard) maximum principle which stipulates that $v_\e$ takes values between $0$ and $1$, and that $u_\e$ is bounded whenever the boundary condition $g$ is. 

\begin{lemma}[Maximum principle]\label{lem:max_princ_for_u}
Let $\O\subset\R^N$ be a bounded open set with Lipschitz boundary and $(u_\e,v_\e)\in H^1(\O)\times [H^1(\O)\cap L^\infty(\O)]$ satisfying \eqref{eq:varformu}-\eqref{eq:varformv}. If $v_\e=1$ on $\partial\O$, then  $ 0 \leq v_\e \leq 1$ a.e.\ in $\O$. In addition, if $u_\e=g$ on $\partial\O$ for a function $g\in H^{\frac12}(\partial\O)\cap L^\infty(\partial\O)$, then $u_\e\in L^\infty(\O)$ and $\|u_\e\|_{L^\infty(\O)} \leq \|g\|_{L^\infty(\partial\O)}$. 
\end{lemma}

\begin{proof}
For {a generic function} \(f \in L^1(\O)\), we set \( f^+:=(f+|f|)/2\) and \(f^-:= (|f|-f)/2\). For simplicity, we drop the subscript \(\e\) in \((u_\e,v_\e)\) and write instead $(u,v)$.

Since $v-1\in H^1_0(\O)\cap L^\infty(\O)$, it follows that  $(v-1)^+ \in H_0^1(\O) \cap L^\infty(\O)$ with $\nabla(v-1)^+=\nabla v {\bf 1}_{\{v\geq 1\}}$.  A classical argument using $(v-1)^+$  as a test function in \eqref{eq:varformv} leads to $v \leq 1$ a.e.\ in $\O$.  
Next, since \(v=1\) on \(\p \O\), we have  $-v^- \in H_0^1(\O)\cap L^\infty(\O)$ and the same argument with \(-v^-\) as test function in \eqref{eq:varformv} shows that $-v=0$ a.e. in $\O$, that is $v \geq 0$ a.e.\ in $\O$.

Now we assume that $u=g$ on $\partial\O$ with $g\in H^{\frac12}(\partial\O)\cap L^\infty(\partial\O)$ and we set  $M:=\|g\|_{L^\infty(\partial\O)}$. 
Since $|g|\leq M$ on $\partial\O$, we have $(u-M)^+\in H^1_0(\O)$ with $\nabla (u-M)^+=\nabla u \mathbf{1}_{\{u\geq M\}}$. 
Using  $(u-M)^+$ as a test function in \eqref{eq:varformu} yields $\nabla (u-M)^+=0$ a.e. in $\O$ which implies that  $(u-M)^+$ is constant. Since $(u-M)^+ \in H^1_0(\O)$, it follows that $(u-M)^+=0$ a.e.\ in $\O$, that is $u \leq M$ a.e.\ in $\O$. The same argument applied to $(u+M)^- \in H^1_0(\O)$ shows that $u\geq -M$ a.e.\ in $\O$, and thus $\|u\|_{L^\infty(\O)}\leq M$.
\end{proof}

Next we study the boundary regularity of a critical point  \( (u_\e,v_\e)\) of the Ambrosio-Tortorelli energy. Our strategy is to use a local reflexion argument to extend \( (u_\e,v_\e)\) across  the boundary. The extension will then satisfy a modified system of PDEs for which we can apply an interior regularity result (similar to {that of} Theorem \ref{prop:int-regularity}). The reflexion argument originates in \cite{Scheven_2006} and follows the arguments in \cite{Dipasquale_Millot_Pisante_2021}. Note that Lemma \ref{lem:max_princ_for_u} and Theorem \ref{thembdregloc}  together with a standard covering argument completes the proof of Theorem \ref{thmmain1}. 

\begin{theorem}\label{thembdregloc}
Let $\O\subset\R^N$ be a bounded open set and $(u_\e,v_\e)\in H^1(\O)\times [H^1(\O)\cap L^\infty(\O)]$ satisfying \eqref{eq:varformu}-\eqref{eq:varformv}.  Assume that in some  ball $B_{4R}(x_0)$ with $x_0\in\p\O$, the boundary portion $\p\O \cap B_{4R}(x_0)$ is of class~$\C^{k\vee 2,1}$ and $(u_\e,v_\e)=(g,1)$ on $\p\O \cap B_{4R}(x_0)$  for a function $g\in \C^{k,\alpha}(\p\O \cap B_{4R}(x_0))$ with  $k\geq 1$ and $\alpha\in(0,1)$. Then $(u_\e,v_\e)\in \C^{k,\alpha}(\overline\O\cap B_{\theta R}(x_0)) \times \C^{k\vee 2,\alpha}(\overline\O\cap B_{\theta R}(x_0))$ for  some constant $\theta\in(0,1)$.
\end{theorem}

\noindent{\it Proof.} 
We start by describing the reflexion method that we use to extend functions across $\p\O$ in a neighborhood of the point $x_0$. We assume that $x_0\in\p\O$ and $R>0$ are fixed, and that the assumption of the theorem is satisfied. Since  \(\p\O \cap B_{4R}(x_0)\) is (at least) of class \(\mathscr C^{2,1}\), we can find a small \(\delta_0\in(0,R/2)\) such that the nearest point projection on \(\p \O\cap B_{4R}(x_0)\), denoted by~\(\boldsymbol{\pi}_\O\), is well-defined and (at least) of class \(\mathscr C^{1,1}\) in a tubular neighborhood of size \(2\delta_0\) of \(\p \O\cap B_{4R}(x_0)\) intersected with $B_{3R}(x_0)$. For \(\delta \in (0,2\delta_0]\), we set
$$
\begin{cases}
U_\delta:=\big\{x \in \R^N : \, {\rm dist}(x,\partial\O)<\delta\big\}\cap B_{3R}(x_0)\,, \\
U_\delta^{\text{in}} := \O \cap U_\delta\,,\\
U_\delta^{\text{ex}}:= U_\delta \setminus \overline{\O}\,.
\end{cases}$$
The {\sl geodesic reflexion} across \(\p \O\cap U_{2\delta_0} \) is denoted by \(\boldsymbol{\sigma}_\O: U_{2\delta_0} \rightarrow \mathbb{R}^N\) and it is defined by 
$$\boldsymbol{\sigma}_\O(x):=2\boldsymbol{\pi}_\O(x)-x \quad \text{ for all }x\in U_{2\delta_0}\,.$$ 
The mapping $\boldsymbol{\sigma}_\O$ is an involutive \(\mathscr C^{1,1}\)-diffeomorphism (onto its image) which satisfies  $\boldsymbol{\sigma}_\O(x)=x$ for all $x\in \p\O\cap U_{2\delta_0}$. Reducing  the value of $\delta_0$ if necessary, we have 
$$\boldsymbol{\sigma}_\O\big(U_\delta^{\text{ex}}\cap {B}_R(x_0)\big)\subset U_\delta^{\text{in}}\cap  B_{2R}(x_0) \text{ and } 
\boldsymbol{\sigma}_\O\big(U_\delta^{\text{in}}\cap {B}_R(x_0)\big)\subset U_\delta^{\text{ex}}\cap   B_{2R}(x_0) \text{ for  }  \delta \in (0,2\delta_0)\,.$$

Next we consider the bounded open set 
\begin{equation}\label{defOmtild}
\widetilde{\O}:=  \big(U_{\delta_0}\cap B_R(x_0)\big)\cup\boldsymbol{\sigma}_\O\big(U_{\delta_0}\cap B_R(x_0)\big)\subset U_{\delta_0}\cap B_{2R}(x_0)\,.
\end{equation}
The mapping $\boldsymbol{\sigma}_\O$ being involutive, we have   
$$\boldsymbol{\sigma}_\O(\widetilde{\O})=\widetilde{\O}\,,\quad\boldsymbol{\sigma}_\O(\widetilde{\O}\cap \O)=\widetilde{\O}\setminus\overline\O\quad\text{and}\quad \boldsymbol{\sigma}_\O(\widetilde{\O}\setminus \overline\O)=\widetilde{\O}\cap\O\,. $$
Differentiating the relation $\boldsymbol{\sigma}_\O(\boldsymbol{\sigma}_\O(x))=x$ yields $D\boldsymbol{\sigma}_\O (x)D\boldsymbol{\sigma}_\O(\boldsymbol{\sigma}_\O(x))= \mathrm{Id}$, and thus  
\begin{equation}\label{invgradsig}
D\boldsymbol{\sigma}_\O(\boldsymbol{\sigma}_\O(x))=\big(D\boldsymbol{\sigma}_\O (x)\big)^{-1}\quad \text{for every } x\in \widetilde{\O}\,.
\end{equation}
For \(x \in \p \O\cap \widetilde\O\), one has  \( (D\boldsymbol{\sigma}_\O (x))^Tv=2{\bf p}_x(v)-v\) for all \(v \in \R^N\), where \({\bf p}_x\) is the orthogonal projection from \(\R^N\) onto the tangent space \(T_x( \p \O)\) to $\partial\O$ at $x$, i.e., $(D\boldsymbol{\sigma}_\O(x))^T$ is the reflexion matrix across the hyperplane $T_x(\partial\O)$. In particular, 
\begin{equation}\label{eq:prop_proj}
(D\boldsymbol{\sigma}_\O(x))^TD \boldsymbol{\sigma}_\O(x)=(D\boldsymbol{\sigma}_\O(x))^T(D\boldsymbol{\sigma}_\O(x))^T={\rm Id} \text{ for every } x \in \p \O\cap \widetilde\O\,.
\end{equation}
Now we define for $x\in\widetilde\O$, 
$$j(x):= \begin{cases}
1 & \text{ if } x \in \widetilde\O\cap\O\,, \\
|\det D\boldsymbol{\sigma}_\O(x)| & \text{ if } x \in \widetilde{\O}\setminus \O\,,
\end{cases}$$
and
$$A(x):= \begin{cases}
{\rm Id} & \text{ if } x \in {\O}\,, \\
j(x)\big[D\boldsymbol{\sigma}_\O(\boldsymbol{\sigma}_\O(x))\big]^TD\boldsymbol{\sigma}_\O(\boldsymbol{\sigma}_\O(x)) & \text{ if } x \in \widetilde{\O} \setminus \O\,.
\end{cases}$$
In view of  \eqref{eq:prop_proj},  $j$ and \(A\) are Lipschitz continuous in \(\widetilde{\O}\) and \(A\) is uniformly elliptic, i.e., there exist two constants \(0<\lambda_\O\leq \Lambda_\O\) such that 
$$ \lambda_\O |\xi|^2 \leq A(x)\xi\cdot \xi \leq \Lambda_\O |\xi|^2 \quad\text{for every } (x,\xi)\in \widetilde{\O}\times \R^N\,.$$

With these geometrical preliminaries, we are now ready to provide the extension of $(u_\e,v_\e)$ to $\widetilde\O$. We define for $x\in \widetilde\O$, 
\begin{equation}\label{defuvchap}
\widehat{u}_\e(x):=
\begin{cases}
u_\e(x) & \text{ if } x \in \O\\
u_\e(\boldsymbol{\sigma}_\O(x)) & \text{ if } x \in \widetilde{\O}\setminus \O
\end{cases}\,,\quad\widehat{v}_\e(x):=
\begin{cases}
v_\e(x) & \text{ if } x \in \O\\
v_\e(\boldsymbol{\sigma}_\O(x)) & \text{ if } x \in \widetilde{\O}\setminus \O
\end{cases}\,,
\end{equation}
and
\begin{equation}\label{eq:def_extensions}
\widetilde{u}_\e(x):= \begin{cases}
u_\e(x) & \text{ if } x \in \O \\
2g(\boldsymbol{\pi}_\O(x))-u_\e(\boldsymbol{\sigma}_\O(x)) &\text{ if } x \in \widetilde{\O}\setminus \O
\end{cases}\,,
\;\; \widetilde{v}_\e(x):= \begin{cases}
v_\e(x)& \text{ if } x \in \O \\
2-v_\e(\boldsymbol{\sigma}_\O(x)) & \text{ if } x \in \widetilde{\O}\setminus \O
\end{cases}\,.
\end{equation}
By the chain rule in Sobolev spaces and the fact that the traces of these functions coincide on both side of $\p\O\cap  \widetilde{\O}$, each one of them belongs to $H^1(\widetilde\O)$. In addition, $\widehat{v}_\e$ and $\widetilde{v}_\e$ also belong to $L^\infty(\widetilde \O)$ since $v_\e\in L^\infty(\O)$. We finally set 
$$\widetilde g:=g\circ\boldsymbol{\pi}_\O \in \C^{1,\alpha}\big(\widetilde\O\big)\,. $$
\vskip3pt 

Now we show that  these extensions satisfy suitable equations in the domain $\widetilde\O$.

\begin{lemma}
We have
\begin{equation}\label{eq:1st_extended_pdes}
-\dive \big( (\eta_\e +\widehat{v}_\e^{\,2})A \nabla \widetilde{u}_\e\big)=-2\dive\big( \mathbf{1}_{\widetilde\O \setminus\overline\O} (\eta_\e+\widehat{v}_\e^{\,2})A\nabla \widetilde g \big) \quad \text{ in }\mathscr{D}^\prime(\widetilde\O)\,,
\end{equation}
and
\begin{equation}\label{eq:2sd_extended_pdes}
-\e\dive(A\nabla\widetilde v_\e)=( \mathbf{1}_{\widetilde\O \cap\O} - \mathbf{1}_{\widetilde\O \setminus\overline\O} ) \Big(\frac{j}{4\e}(1-\widehat{v}_\e)-\big(A\nabla \widehat{u}_\e\cdot\nabla\widehat u_\e\big)\widehat{v}_\e \Big) \quad \text{ in }\mathscr{D}^\prime(\widetilde\O)\,.
\end{equation}
\end{lemma}

\begin{proof}
Again, for simplicity, we drop the subscript \(\e\).  We fix an arbitrary test function $\varphi\in\mathscr{D}(\widetilde\O)$, and we define for $x\in \widetilde\O$ the {\sl symmetric} and {\sl anti-symmetric} parts of $\varphi$, 
$$\varphi^s(x):=\frac{1}{2}\big(\varphi(x)+\varphi\circ \boldsymbol{\sigma}_\O(x)\big)\,, \quad \varphi^a(x):=\frac{1}{2}\big(\varphi(x)- \varphi\circ \boldsymbol{\sigma}_\O(x)\big)\,.$$
The functions $\varphi^s$ and $\varphi^a$ belong to $\C^{1,1}(\widetilde\O)$ and, by construction, $\varphi^s\circ \boldsymbol{\sigma}_\O=\varphi^s$ and $\varphi^a\circ \boldsymbol{\sigma}_\O=-\varphi^a$.
\vskip3pt

\noindent{\it Step 1: proof of \eqref{eq:1st_extended_pdes}.} We start with the identity 
\begin{multline}\label{calceqsym1}
\int_{\widetilde\O\setminus\overline\O} (\eta_\e+\widehat v^{\,2})(A\nabla \widetilde u)\cdot \nabla\varphi^s\dd x=-\int_{\widetilde\O\setminus\overline\O} (\eta_\e+\widehat v^{\,2})\big( A\nabla ( u\circ \boldsymbol{\sigma}_\O) \big)\cdot  \nabla(\varphi^s\circ  \boldsymbol{\sigma}_\O)\dd x\\
+ 2\int_{\widetilde\O\setminus\overline\O} (\eta_\e+\widehat v^{\,2})(A\nabla \widetilde g)\cdot \nabla\varphi^s\dd x\,.
\end{multline}
Using relation \eqref{invgradsig} and changing variables yields 
\begin{align}
\nonumber\int_{\widetilde\O\setminus\overline\O} (\eta_\e+\widehat v^{\,2})\big( A\nabla ( u\circ \boldsymbol{\sigma}_\O) \big)\cdot  \nabla(\varphi^s\circ  \boldsymbol{\sigma}_\O)\dd x&= \int_{\widetilde\O\setminus\overline\O} (\eta_\e+\widehat v^{\,2})\nabla u\big(\boldsymbol{\sigma}_\O\big)\cdot  \nabla\varphi^s\big(\boldsymbol{\sigma}_\O\big)j \dd x\\
\label{calceqsym2}&=\int_{\widetilde\O\cap\O} (\eta_\e+ \widehat v^{\,2})\big(A\nabla \widetilde u\big)\cdot  \nabla\varphi^s\dd x\,.
\end{align}

Therefore, combining \eqref{calceqsym1} and \eqref{calceqsym2} yields 
$$\int_{\widetilde\O} (\eta_\e+\widehat v^{\,2})(A\nabla \widetilde u)\cdot \nabla\varphi^s\dd x=2\int_{\widetilde\O\setminus\overline\O} (\eta_\e+\widehat v^{\,2})(A\nabla \widetilde g)\cdot \nabla\varphi^s\dd x\,.$$
In the same way, we have  
\begin{multline*}
\int_{\widetilde\O\setminus\overline\O} (\eta_\e+\widehat v^{\,2})(A\nabla \widetilde u)\cdot \nabla\varphi^a\dd x=\int_{\widetilde\O\setminus\overline\O} (\eta_\e+\widehat v^{\,2})\big( A\nabla ( u\circ \boldsymbol{\sigma}_\O) \big)\cdot  \nabla(\varphi^a\circ  \boldsymbol{\sigma}_\O)\dd x\\
+ 2\int_{\widetilde\O\setminus\overline\O} (\eta_\e+\widehat v^{\,2})(A\nabla \widetilde g)\cdot \nabla\varphi^a\dd x\,.
\end{multline*}
Since $\varphi^a=0$ on $\widetilde\O\cap\p\O$, we have $\varphi^a\in H^1_0(\widetilde\O\cap\O)$. Hence, we can use the first equation in \eqref{eq:PDE}  to infer that 
\begin{align*}
\int_{\widetilde\O\setminus\overline\O} (\eta_\e+\widehat v^{\,2})\big( A\nabla ( u\circ \boldsymbol{\sigma}_\O) \big)\cdot  \nabla(\varphi^a\circ  \boldsymbol{\sigma}_\O)\dd x&= \int_{\widetilde\O\setminus\overline\O} (\eta_\e+\widehat v^{\,2})\nabla u\big(\boldsymbol{\sigma}_\O\big)\cdot  \nabla\varphi^a\big(\boldsymbol{\sigma}_\O\big)j \dd x\\
&=\int_{\widetilde\O\cap\O} (\eta_\e+  v^{\,2})\nabla  u\cdot  \nabla\varphi^a\dd x= 0 \,.
\end{align*}
Consequently, 
$$\int_{\widetilde\O} (\eta_\e+\widehat v^{\,2})(A\nabla \widetilde u)\cdot \nabla\varphi\dd x=2\int_{\widetilde\O\setminus\overline\O} (\eta_\e+\widehat v^{\,2})(A\nabla \widetilde g)\cdot \nabla\varphi\dd x\,,$$
and \eqref{eq:1st_extended_pdes} follows. 
\vskip3pt

\noindent{\it Step 2: proof of \eqref{eq:2sd_extended_pdes}.}  We proceed as above, starting with 

$$\e\int_{\widetilde\O\setminus\overline\O} (A\nabla \widetilde v)\cdot \nabla\varphi^s\dd x=-\e\int_{\widetilde\O\setminus\overline\O} \big( A\nabla ( v\circ \boldsymbol{\sigma}_\O) \big)\cdot  \nabla(\varphi^s\circ  \boldsymbol{\sigma}_\O)\dd x=-\e\int_{\widetilde\O\cap\O} (A\nabla \widetilde v)\cdot \nabla\varphi^s\dd x\,,$$
which yields 
\begin{equation}\label{calceqantsym1}
\e\int_{\widetilde\O} (A\nabla \widetilde v)\cdot \nabla\varphi^s\dd x=0\,.
\end{equation}
On the other hand, 
\begin{equation}\label{calceqantsym2}
\e\int_{\widetilde\O\setminus\overline\O} (A\nabla \widetilde v)\cdot \nabla\varphi^a\dd x=\e\int_{\widetilde\O\setminus\overline\O} \big( A\nabla ( v\circ \boldsymbol{\sigma}_\O) \big)\cdot  \nabla(\varphi^a\circ  \boldsymbol{\sigma}_\O)\dd x=\e\int_{\widetilde\O\cap\O} \nabla v\cdot \nabla\varphi^a\dd x\,.
\end{equation}
Since $\varphi^a\in H^1_0(\widetilde\O\cap\O)$, we can apply the second equation in \eqref{eq:PDE} to deduce that 
\begin{equation}\label{calceqantsym3}
\e\int_{\widetilde\O\cap\O} \nabla v\cdot \nabla\varphi^a\dd x= -\int_{\widetilde\O\cap\O}|\nabla u|^2v\varphi^{a}\dd x+\frac{1}{4\e}\int_{\widetilde\O\cap\O}(1-v)\varphi^a\dd x\,.
\end{equation}
Summing up \eqref{calceqantsym1},  \eqref{calceqantsym2}, and  \eqref{calceqantsym3}, and using that \(\boldsymbol{\sigma}_\O\) is an involution leads to 
\begin{multline*}
\e\int_{\widetilde\O} (A\nabla \widetilde v)\cdot \nabla\varphi\dd x= -\int_{\widetilde\O\cap\O}|\nabla \widehat{u}|^2\widehat{v}\varphi\dd x+\frac{1}{4\e}\int_{\widetilde\O\cap\O}(1-\widehat{v})\varphi\dd x\\
 +\int_{\widetilde\O\cap\O}|\nabla \widehat{u}|^2\widehat{v}\varphi\circ\boldsymbol{\sigma}_\O\dd x-\frac{1}{4\e}\int_{\widetilde\O\cap\O}(1-\widehat{v})\varphi\circ\boldsymbol{\sigma}_\O\dd x\,.
 \end{multline*}
Changing variables in the two last integrals, we obtain
$$\e\int_{\widetilde\O} (A\nabla \widetilde v)\cdot \nabla\varphi\dd x= -\int_{\widetilde\O}( \mathbf{1}_{\widetilde\O \cap\O} - \mathbf{1}_{\widetilde\O \setminus\overline\O} )\big((A\nabla \widehat{u})\cdot\nabla\widehat u\big)\widehat{v}\varphi\dd x+\frac{1}{4\e}\int_{\widetilde\O}( \mathbf{1}_{\widetilde\O \cap\O} - \mathbf{1}_{\widetilde\O \setminus\overline\O} )(1-\widehat{v})\varphi j \dd x$$
and \eqref{eq:2sd_extended_pdes} follows. 
\end{proof}

We now provide a general regularity result generalizing the argument used in the proof of the interior regularity.

\begin{lemma}\label{lem:Morrey_estimates}
Let \(A:B_R  \rightarrow \mathbf M^{N \times N}_{\rm sym}\) be a Lipschitz field of symmetric \(N \times N\) matrices which is uniformly elliptic (i.e., there exist \(0<\lambda <\Lambda\) such that \( \lambda|\xi|^2\leq A(x)\xi\cdot\xi\leq \Lambda |\xi|^2\) for all \((x,\xi)\in B_R \times\R^N\)), and \(f \in L^1(B_R)\) satisfying 
\begin{equation}\label{eq:hypf}
 \sup_{B_\varrho(x_0) \subset B_R} \frac{1}{\varrho^{N-2+\gamma}} \int_{B_\varrho(x_0)}|f| \dd x <\infty\,,
\end{equation}
for some \(\gamma \in (0,2]\). If \(z \in H^1(B_R)\cap L^\infty(B_R)\) solves 
\begin{equation}\label{eq:pdev}
-\dive(A\nabla z)= f \text{ in } \mathscr{D}^\prime(B_R)\,,
\end{equation}
then 
\begin{equation}\label{eq:concl}
\sup_{x_0 \in \overline{B}_{R/2}, \, \varrho\in(0,R/2]} \frac{1}{\varrho^{N-2+2\alpha}} \int_{B_\varrho(x_0)}|\nabla z|^2 \dd x <\infty\,,
\end{equation}
and \(z \in \C^{0,\alpha}(B_{R/2})\) for every $\alpha\in(0,\gamma/2)$.
\end{lemma}

\begin{proof}
Throughout the proof, we fix an exponent $\alpha\in(0,\gamma/2)$ and we set $\beta:=\gamma-2\alpha>0$. We also denote by $K$ an upper bound for $\|z\|_{L^\infty(B_R)}\), and by $M$ an upper bound for \eqref{eq:hypf}. Then $C>0$ shall stand for a constant (which may vary from line to line)  depending only on $N$, $\alpha$, $\gamma$, $\lambda$, $\Lambda$, $K$, $M$, and the Lipschitz constant of $A$. 
\vskip3pt

Let us fix $x_0\in \overline{B}_{R/2}$ and $\varrho\in(0,R/2]$ arbitrary, and consider \(w \in H^1(B_\varrho(x_0))\)  the unique (weak) solution of
\begin{equation}\label{eq:pde_on_v}
\left\{
\begin{array}{rcll}
-\dive(A(x_0)\nabla w)&=& 0 & \text{ in } B_\varrho(x_0)\,, \\
w&=&z & \text{ on } \p B_\varrho(x_0)\,.
\end{array}
\right.
\end{equation}
Recalling that 
$$ \int_{B_\varrho(x_0)} A(x_0)\nabla w\cdot\nabla w\dd x \leq  \int_{B_\varrho(x_0)} A(x_0)\nabla \bar w\cdot\nabla \bar w\dd x\quad\text{for every $\bar w\in w+H^1_0(B_\varrho(x_0))\,,$}$$
we have 
\begin{equation}\label{eq:wmin}
 \lambda\int_{B_\varrho(x_0)} |\nabla w|^2 \dd x \leq \int_{B_\varrho(x_0)} A(x_0)\nabla w \cdot \nabla w \dd x \leq  \int_{B_\varrho(x_0)} A(x_0) \nabla z \cdot \nabla z \dd x\leq \Lambda\int_{B_\varrho(x_0)} |\nabla z|^2 \dd x.
 \end{equation}
Moreover, according to the maximum principle, \(\|w\|_{L^\infty(B_\varrho(x_0))} \leq \|z\|_{L^\infty(B_\varrho(x_0))}\leq K\).  

First, we infer from the triangle inequality, 
\begin{align}
\nonumber \left( \frac{1}{(\varrho/2)^{N-2+2\alpha}}\int_{B_{\frac{\varrho}{2}}(x_0)} A\nabla z\cdot \nabla z \dd x\right)^{\frac12} & \leq 
\left( \frac{1}{(\varrho/2)^{N-2+2\alpha}}  \int_{B_{\frac{\varrho}{2}}(x_0)} A \nabla w \cdot \nabla w  \dd x\right)^{\frac12} \nonumber\\
\nonumber  &\qquad +\left(\frac{1}{(\varrho/2)^{N-2+2\alpha}}\int_{B_\frac{\varrho}{2}(x_0)} A(\nabla z-\nabla w)\cdot (\nabla z-\nabla w) \dd x \right)^{\frac12}\\
& \leq  \left( \frac{1}{(\varrho/2)^{N-2+2\alpha}}  \int_{B_{\frac{\varrho}{2}}(x_0)} A \nabla w \cdot \nabla w \dd x\right)^{\frac12}\nonumber\\
\label{eq:1657}  & \qquad+C\left(\frac{1}{\varrho^{N-2+2\alpha}} \int_{B_\varrho(x_0)} A(\nabla z-\nabla w)\cdot (\nabla z-\nabla w )\dd x\right)^{\frac12}\,.
\end{align}
We start by estimating the first term in the right-hand-side of \eqref{eq:1657} using \eqref{eq:wmin}, and the fact that $A$ is Lipschitz continuous and uniformly elliptic. It yields 
\begin{multline}\label{eq:estimate_I}
  \int_{B_{\frac{\varrho}{2}}(x_0)} A \nabla w \cdot \nabla w  \dd x 
 \leq  (1+C\rho) \int_{B_\varrho(x_0)} A(x_0)\nabla w \cdot \nabla w \dd x \\
\leq  (1+C\rho) \int_{B_\varrho(x_0)} A(x_0)\nabla z \cdot \nabla z \dd x \leq (1+C\rho) \int_{B_\varrho(x_0)} A\nabla z \cdot \nabla z \dd x\,.
\end{multline}
To estimate the second term in the right-hand-side of \eqref{eq:1657}, we make use of equation \eqref{eq:pdev} to write
\begin{align*}
\int_{B_\varrho(x_0)} A (\nabla z-\nabla w) \cdot (\nabla z-\nabla w) \dd x&= \int_{B_\varrho(x_0)} A \nabla z\cdot \nabla( z- w) \dd x -\int_{B_\varrho(x_0)} A \nabla w \cdot \nabla( z- w) \dd x \\
&=\int_{B_\varrho(x_0)} f (z-w) \dd x -\int_{B_\varrho(x_0)} A \nabla w\cdot\nabla ( z- w) \dd x\,.
\end{align*}
Using assumption \eqref{eq:hypf} on \(f\), we infer that
\begin{equation}\label{eq:estimate_Ibis}
 \left|\frac{1}{\varrho^{N-2+2\alpha}} \int_{B_\varrho(x_0)} f(z-w) \dd x\right| \leq 2KM \varrho^{\beta}\,.
 \end{equation}
On the other hand,  Equation \eqref{eq:pde_on_v} satisfied by \(w\) implies that 
\begin{align}\label{eq:estimate_II}
\int_{B_\varrho(x_0)}A\nabla w\cdot \nabla(z- w) \dd x &= \int_{B_\varrho(x_0)} (A-A(x_0))\nabla w \cdot \nabla(z- w) \dd x\nonumber\\
& \leq  C \varrho \int_{B_\varrho(x_0)} \big[ |\nabla w|^2+|\nabla z|^2 \big] \dd x \nonumber\\
& \leq  C \varrho \int_{B_\varrho(x_0)} |\nabla z|^2 \dd x \leq C \varrho \int_{B_\varrho(x_0)} A\nabla z\cdot \nabla z \dd x\,,
\end{align}
where we used again \eqref{eq:wmin} together with the ellipticity of $A$.

Gathering \eqref{eq:1657}, \eqref{eq:estimate_I}, \eqref{eq:estimate_Ibis} and \eqref{eq:estimate_II}, we get that
\begin{align*}
\left(\frac{1}{(\varrho/2)^{N-2+2\alpha}}\int_{B_\frac{\varrho}{2}(x_0)} A\nabla z\cdot \nabla z \dd x\right)^{1/2} & \leq (1+C\sqrt\varrho)\left(\frac{1}{\varrho^{N-2+2\alpha}}\int_{B_\varrho(x_0)} A\nabla z\cdot \nabla z \dd x\right)^{1/2} +C \varrho^{\beta/2}\,.
\end{align*}
We now choose \(\varrho=\varrho_k=2^{-(k+1)}R\) for \(k \in \mathbf{N}\), and we obtain 
\begin{multline*}
\left(\frac{1}{\varrho_{k+1}^{N-2+2\alpha}} \int_{B_{\varrho_{k+1}}(x_0)} A\nabla z\cdot\nabla z \dd x\right)^{1/2} \leq  (1+C\sqrt{R} \,2^{-(k+1)/2})\left(\frac{1}{\varrho_k^{N-2+2\alpha}}\int_{B_{\varrho_k}(x_0)}A\nabla z\cdot \nabla z \dd x\right)^{1/2}\\
+C R^{\beta/2}2^{-\beta(k+1)/2}\,.
\end{multline*}
Next, we observe that if \((\theta_k)_{k\in \mathbf N}\), \((\sigma_k)_{k\in \mathbf N}\), and \((y_k)_{k\in \mathbf N}\) are real sequences such that \(\theta_k \in (1,\infty)\), \(\theta:=\prod_{k=0}^{\infty} \theta_k <\infty\), \(\sigma_k \in (0,\infty)\), $\sigma:=\sum_{k=0}^\infty \sigma_k<\infty$, and satisfying \(y_{k+1} \leq \theta_k y_k+\sigma_k\) for all \( k \in \mathbf N\), then \(y_k \leq \theta (y_0+\sigma)\).
Applying this principle with  
$$y_k= \left( \frac{1}{\varrho_k^{N-2+2\alpha}}\int_{B_{\varrho_{k}}(x_0)} A\nabla z\cdot \nabla z \dd x \right)^{1/2}\,, \quad \theta_k=1+C\sqrt{R}\,2^{-(k+1)/2}\,, \quad \sigma_k=CR^{\beta/2}2^{-\beta(k+1)/2}\,,$$
yields
$$\frac{1}{\varrho_k^{N-2+2\alpha}} \int_{B_{\varrho_k}(x_0)} A\nabla z\cdot \nabla z \dd x\leq Ce^{C\sqrt{R}}\left( \frac{1}{R^{N-2+2\alpha}}\int_{B_R} A\nabla z\cdot \nabla z \dd x+R^{\beta/2}\right) \quad \text{ for all } k \in \mathbf{N}\,$$
(we have also used the elementary estimates $\theta\leq e^{C\sqrt{R}}$ and $\sigma \leq CR^{\beta/2}$). 
Since for all \( \varrho \in (0,R/2]\), there exists a unique \(k \in \mathbf{N}\) such that \(\varrho_{k+1}<\varrho\leq \varrho_{k}\) and \(\frac{1}{\varrho}<\frac{1}{\varrho_{k+1}}\leq \frac{2}{\varrho_k}\), we conclude that
\[ \frac{1}{\varrho^{N-2+2\alpha}} \int_{B_\varrho(x_0)} A\nabla z\cdot \nabla z \dd x \leq Ce^{C\sqrt{R}}\left( \frac{1}{R^{N-2+2\alpha}}\int_{B_R} A\nabla z\cdot \nabla z\dd x+R^{\beta/2}\right)\quad\forall \varrho\in(0,R/2]\,.\]
Finally, by ellipticity of $A$ and the arbitrariness of $x_0$, we conclude that \eqref{eq:concl} holds with 
\[ \frac{1}{\varrho^{N-2+2\alpha}} \int_{B_\varrho(x_0)} |\nabla z|^2\dd x \leq Ce^{C\sqrt{R}}\left( \frac{1}{R^{N-2+2\alpha}}\int_{B_R} |\nabla z|^2 \dd x+R^{\beta/2}\right)\quad\forall \varrho\in(0,R/2]\,,\;\forall x_0\in \overline{B}_{R/2}\,.\]
By Morrey's Theorem (see e.g. \cite[Theorem 5.7]{Giaquinta_Martinazzi_2012}), it then follows that $v\in \C^{0,\alpha}(B_{R/2})$. 
\end{proof}

We are now ready to prove the boundary regularity result in Theorem \ref{thembdregloc}. 

\begin{proof}[Proof of Theorem \ref{thembdregloc} completed]
 We consider $(u_\e,v_\e)\in H^1(\O)\times [H^1(\O)\cap L^\infty(\O)]$ satisfying \eqref{eq:varformu}-\eqref{eq:varformv}, and we consider the extensions $\widehat u_\e$, $\widetilde u_\e$, $\widehat v_\e$, and $\widetilde v_\e$ to the domain $\widetilde \O$ (depending on $x_0$ and $R$) provided by \eqref{defuvchap}-\eqref{eq:def_extensions} and \eqref{defOmtild}. Again, for simplicity, we drop the subscript \(\e\).

We first improve the regularity of  \(\widetilde{u}\) which satisfies \eqref{eq:1st_extended_pdes}. We aim to apply the De-Giorgi-Nash-Moser Theorem to infer that \(\widetilde{u}\) is locally H\"older continuous in \(\widetilde{\O}\) and that a suitable Morrey estimate holds for \(\nabla \widehat {u}\). Since equation \eqref{eq:1st_extended_pdes} is linear with respect to \(\widetilde{u}\), we first observe that 
$$\int_{\widetilde\O}f(x,\nabla \widetilde u)\dd x \leq \int_{\widetilde\O}f(x,\nabla  w)\dd x\quad\text{for all $w\in H^1(\O)$ such that ${\rm supp}(w-\widetilde u)\subset\widetilde\O$}\,,$$
with  
$$f(x,\xi):=\frac{1}{2}(\eta_\e +\widehat{v}^{\,2}_\e(x))A(x)\xi\cdot \xi -  h(x)\cdot \xi \quad \text{ for a.e. } x \in \widetilde{\O} \text{ and all }\xi\in \R^N\,,$$
and
$$h:=2\, \mathbf{1}_{\widetilde\O \setminus\overline\O} (\eta_\e+\widehat{v}^{\,2})A\nabla \widetilde g \in L^\infty(\widetilde \O;\R^N)\,.$$
The function $f$ is  a Carath\'eodory function, and since \(A\) is uniformly elliptic and the functions $\widehat v$ and $h$ are essentially bounded, we can find positive constants $c_1$, $c_2$, and $c_3$ such that
\begin{equation*}
c_1|\xi|^2-c_3\leq f(x,\xi) \leq c_2|\xi|^2+c_3 \quad \text{ for a.e. }x \in \widetilde\O \text{ and all } \xi\in \R^N\,.
\end{equation*}
Hence we can  apply the De Giorgi-Nash-Moser Theorem (see \cite[Theorems 7.5 and 7.6]{Giusti_2003}) to deduce the existence of some \(\beta \in (0,1)\) such that \(\widetilde{u} \in \C^{0,\beta}_{\rm loc}(\widetilde\O)\). From  \cite[Theorem 7.7]{Giusti_2003} and \cite[Lemma 5.13]{Giaquinta_Martinazzi_2012}, we also obtain the Morrey estimate 
 \begin{equation}\label{eq:estimate_nabla_hat_u}
 \sup_{B_\varrho(x_*) \subset B_{\delta_0/2}(x_0)} \frac{1}{\varrho^{N-2+2\beta}}\int_{B_\varrho(x_*)}|\nabla \widetilde{u}|^2 \dd x < \infty\,, 
 \end{equation}
since ${B}_{\delta_0}(x_0)\subset \widetilde\O$. 
 
Next we consider the  equation \eqref{eq:2sd_extended_pdes} satisfied by \(\widetilde{v}\) restricted to $B_{\delta_0/2}(x_0)$, that we write 
$$-\dive(A \nabla \widetilde{v})= H \quad \text{ in }\mathscr{D}^{\prime}(B_{\delta_0/2}(x_0))\,,$$
with
$$H:= \frac{1}{\e}( \mathbf{1}_{\widetilde\O \cap\O} - \mathbf{1}_{\widetilde\O \setminus\overline\O} ) \Big(\frac{j}{4\e}(1-\widehat{v})-\big(A\nabla \widehat{u}\cdot\nabla\widehat u\big)\widehat{v} \Big) \,.$$
Since $\widehat u=2\widetilde g - \widetilde u$ and $\nabla\widetilde g$, $A$, $j$, and $\widehat v$ are essentially bounded, we infer from \eqref{eq:estimate_nabla_hat_u} that
$$\sup_{B_\varrho(x_*) \subset B_{\delta_0/2}(x_0)} \frac{1}{\varrho^{N-2+2\beta}}\int_{B_\varrho(x_*)}|H| \dd x < \infty\,.$$

Applying Lemma \ref{lem:Morrey_estimates}, we deduce that \( \widetilde{v} \in \C^{0,\gamma}(B_{\delta_0/4}(x_0))\) for every $\gamma\in(0,\beta)$. 
In particular, we have $v \in \C^{0,\gamma}(\overline\O\cap B_{\delta_0/4}(x_0))$ for every $\gamma\in(0,\beta)$.  
Using the equation \eqref{eq:PDE} satisfied by $u$ together with the (up to the boundary) Schauder estimate (see \cite[Theorem 5.21]{Giaquinta_Martinazzi_2012}), we obtain that $u \in \C^{1,\gamma}(\overline \O\cap B_{\delta_0/5}(x_0))$ for every $\gamma\in(0,\beta)$. Then, in view of the equation \eqref{eq:PDE} satisfied by $v$, and owing to the classical elliptic regularity at the boundary, we obtain $v \in \C^{2,\gamma}(\overline \O\cap B_{\delta_0/6}(x_0))$ for every $\gamma\in(0,\beta)$. Back to the equation in  \eqref{eq:PDE} satisfied by $u$, elliptic regularity at the boundary now tells us that $u \in \C^{1,\alpha}(\overline \O\cap B_{\delta_0/7}(x_0))$ in the case $g\in \C^{1,\alpha}(\partial\O\cap B_{4R}(x_0))$, and in turn  $v \in \C^{2,\alpha}(\overline \O\cap B_{\delta_0/8}(x_0))$ still by \eqref{eq:PDE}.  If $g\in \C^{2,\alpha}(\partial\O\cap B_{4R}(x_0))$, then $v \in \C^{2,\alpha}(\overline \O\cap B_{\delta_0/8}(x_0))$ and once again, elliptic boundary regularity implies that $u \in \C^{2,\alpha}(\overline \O\cap B_{\delta_0/9}(x_0))$.

If $\partial\O\cap B_{4R}(x_0)$ is of class $\C^{k,\alpha}$ and $g\in \C^{k,\alpha}(\partial\O\cap B_{4R}(x_0))$ with $k\geq 3$, one can iterate the preceding argument using elliptic boundary regularity to conclude that $u$ and $v$ belong to $\C^{k,\alpha}(\overline \O\cap B_{\theta R}(x_0))$ for some $\theta >0$ small enough. 
\end{proof}

\section{Compactness results}\label{sec:compactness}

We start by recalling a weak compactness result, in the spirit of the compactness argument in the $\Gamma$-convergence analysis, under the only assumption of uniform energy bound \eqref{eq:hypnrj}. The result is a direct application of the standard lower bound inequality considering the extension of a pair $(u,v)\in  \mathcal A_g(\O)$ to a larger bounded open set $\O^\prime\supset\overline\O$ of the form $(u,v)=(g^\prime,1)$ in $\O^\prime\setminus\O$ for some arbitrary extension $g^\prime\in H^1(\O^\prime)\cap L^\infty(\O^\prime)$  of $g$. 

\begin{proposition}[Weak compactness]\label{prop:lsc}
Let $\O \subset \R^N$ be a bounded open set with Lipschitz boundary, $g\in H^{\frac12}(\partial\O)\cap L^\infty(\partial\O)$, and $\e_n\rightarrow 0^+$ be an arbitrary sequence. Assume that $(u_n,v_n):=(u_{\e_n},v_{\e_n}) \in \mathcal A_g(\O)$ satisfies $0\leq v_n\leq 1$ a.e. in $\O$, $\|u_n\|_{L^\infty(\O)}\leq \|g\|_{L^\infty(\partial\O)}$, and 
the uniform energy bound $\sup_n AT_{\e_n}(u_n,v_n)<\infty$. There exist a (not relabelled) subsequence  and $u_* \in SBV^2(\O) \cap L^\infty(\O)$ such that 
$$\begin{cases}
(u_n,v_n) \to (u_*,1) &\text{ strongly in } {[L^2(\O)]^2}\,,\\
v_n \nabla u_n \wto \nabla u_* & \text{ weakly in }L^2(\O;\R^N)\,.
\end{cases}
$$
Moreover,
\begin{equation}\label{eq:lsc-bulk}
\int_\O |\nabla u_*|^2\dd x \leq \liminf_{n \to \infty} \int_\O v_n^2 |\nabla u_n|^2\dd x \leq \liminf_{n \to \infty} \int_\O (\eta_{\e_n}+v_n^2) |\nabla u_n|^2\dd x\,,
\end{equation}
and
\begin{align}\label{eq:lsc-phase_field}
\HH^{N-1}(J_{u_*} \cup (\partial\O \cap \{u_* \neq g\})) & \leq  \liminf_{n\to\infty} \int_{\O} (1-v_n)|\nabla v_n|\dd x\nonumber\\
& \leq  \liminf_{n \to \infty} \int_\O\left(\e_n|\nabla v_n|^2 +\frac{(v_n-1)^2}{4\e_n}\right) \dd x\,.
\end{align}
\end{proposition}

The regularity of solutions {established in Theorem \ref{thmmain1}} allows us to prove that critical points of  the Ambrosio-Tortorelli functional {satisfies} a Noether type conservation law, which is  the starting point of their asymptotic analysis.

\begin{proposition}\label{prop:Noether_conservation}
Let $\O \subset \R^N$ be a bounded open set with boundary of class $\mathscr C^{2,1}$ and $g\in \mathscr C^{2,\alpha}(\partial\O)$ for some $\alpha \in (0,1)$. If \( (u_\e,v_\e) \in \mathcal A_g(\O)\) is a critical point of $AT_\e$, then for all $X \in \mathscr C^1_c(\R^N;\R^N)$, 
\begin{multline}\label{eq:first-var}
\int_\O  (\eta_\e+v_\e^2) \left[2 \nabla u_\e \otimes \nabla u_\e-|\nabla u_\e|^2 {\rm Id} \right]: DX \dd x\\
+ \int_\O\left[2\e \nabla v_\e \otimes \nabla v_\e - \left(\e|\nabla v_\e|^2 +\frac{(v_\e-1)^2}{4\e} \right){\rm Id}\right]: DX \dd x\\
=\int_{\partial\O} \big[(\eta_\e+1)(\partial_\nu u_\e)^2+\e (\partial_\nu v_\e)^2-(\eta_\e+1) |\nabla_\tau g|^2\big] (X \cdot \nu_\O) \dd\HH^{N-1}\\
+2 (\eta_\e+1)\int_{\partial \O} (\partial_\nu u_\e)(X_\tau \cdot\nabla_\tau g) \dd\HH^{N-1}\, ,
\end{multline}
where $X_\tau:=X-(X\cdot\nu_\O)\nu_\O$ is the tangential part of $X$, and $\nabla_\tau g$ is the tangential gradient of $g$ on $\partial\O$. 
\end{proposition}

\begin{proof}

Let us fix an arbitrary $X \in \mathscr C^1_c(\R^N;\R^N)$. By Theorem \ref{thmmain1}, $(u_\e,v_\e) \in [\mathscr C^{2,\alpha}(\overline \O)]^2$ and  \eqref{eq:PDE} is satisfied in the classical sense. Multiplying the first equation of \eqref{eq:PDE} by $X\cdot \nabla u_\e$ (which belongs to \(\mathscr C^{1}(\overline{\O}\)))  and by integration by parts, a stantard computation yields
\begin{multline}\label{form1}
0=\int_\O (\eta_\e+v_\e^2)\left[(\nabla u_\e\otimes \nabla u_\e) -\frac12 |\nabla u_\e|^2{\rm Id}\right]: DX\dd x-\int_\O v_\e (X\cdot \nabla v_\e)|\nabla u_\e|^2\dd x\\
-\frac{\eta_\e+1}{2}\int_{\partial\O} (\partial_\nu u_\e)^2 (X\cdot \nu_\O)\dd\HH^{N-1}+\frac{\eta_\e+1}{2}\int_{\partial\O} |\nabla_\tau g|^2(X\cdot \nu_\O)\dd\HH^{N-1}\\
-(\eta_\e+1)\int_{\partial \O} (\partial_\nu u_\e)(X_\tau\cdot\nabla_\tau g) \dd\HH^{N-1}\,.
\end{multline}
Similarly, multiplying the second equation  in \eqref{eq:PDE} by $X \cdot \nabla v_\e$  (which belongs to \( \C^{1}(\overline{\O})\)) and performing similar integration by parts leads to
\begin{multline}\label{form2}
0=\e\int_\O\left[ (\nabla v_\e \otimes \nabla v_\e) -\frac{1}{2}\bigg(|\nabla v_\e|^2+\frac{1}{4\e}(v_\e-1)^2\bigg){\rm Id}\right] : DX\dd x+\int_\O v_\e (X\cdot \nabla v_\e)|\nabla u_\e|^2\dd x\\
 - \frac{\e}{2} \int_{\partial\O}(\p_\nu v_\e)^2 (X\cdot \nu_\O) \dd \HH^{N-1}\,,
\end{multline}
since $ v_\e=1$ on $\partial\O$. 
Then the conclusion  follows summing up \eqref{form1} and \eqref{form2}.
\end{proof}

\begin{remark}{\rm 
The fact that critical points \((u_\e,v_\e)\) enjoy the higher regularity \([\C^{2,\alpha}(\overline{\O})]^2\) allows one to obtain a strong form of the conservative equations for \( (u_\e,v_\e)\).  In particular, some information on the boundary are recovered since the  vector field $X$ does not need to be tangential on $\partial\O$. This additional information will be instrumental in Section \ref{sec:convergence} to characterize the boundary term occurring in the first inner variation of the Mumford-Shah functional. }
\end{remark}

Owing to the previous results, we get the following property for the weak limit $u_*$ as $\e\to 0$ of a converging sequence of critical points the Ambrosio-Tortorelli functional. 

\begin{lemma}\label{lem:conv-bdry}
Let $\O \subset \R^N$ be a bounded open set with boundary of class $\C^{2,1}$ and $g\in \C^{2,\alpha}(\partial\O)$ for some $\alpha \in (0,1)$. Along a sequence $\e\to 0^+$, 
let \( (u_\e,v_\e) \in\mathcal A_g(\O)\) be a critical point of $AT_\e$ satisfying the uniform energy bound \eqref{eq:hypnrj} and the conclusion of Proposition \ref{prop:lsc}. If $u_*$ denotes the weak limit of $u_\e$ as $\e\to0$, then $\nabla u_* \in L^2(\O;\R^N)$ satisfies $\dive(\nabla u_*)=0$ in $\mathscr D'(\O)$, its normal trace $\nabla u_*\cdot\nu_\O$ belongs to $L^2(\partial\O)$, and $\partial_\nu u_\e \wto \nabla u_*\cdot \nu_\O$ weakly in $L^2(\partial\O)$  as $\e\to0$. 
Moreover, up to a subsequence, there exists a nonnegative Radon measure $\lambda_* \in \M^+(\partial\O)$ such that
$$[(\partial_\nu u_\e)^2 +\e(\partial_\nu v_\e)^2]\HH^{N-1}\res\partial\O \wsto \lambda_* \quad \text{ weakly* in } \M(\partial\O)\,.$$
\end{lemma}

\begin{proof}
We first claim that $(\eta_\e+v_\e^2)\nabla u_\e \wto \nabla u_*$ weakly in $L^2(\O;\R^N)$. Indeed, on the one hand we have
\begin{equation}\label{eq:conv_eta_eps_nablau}
\|\eta_\e\nabla u_\e\|_{L^2(\O;\R^N)} \leq \sqrt{\eta_\e} \left\|\sqrt{\eta_\e+v_\e^2}\nabla u_\e\right\|_{L^2(\O;\R^N)} \leq C\sqrt{\eta_\e}\to 0\,,
\end{equation}
and on the other hand, for all $\varphi \in \C^\infty_c(\O;\R^N)$,
$$\left| \int_\O v_\e\nabla u_\e \cdot \varphi\dd x - \int_\O v^2_\e\nabla u_\e \cdot \varphi\dd x\right| \leq \|\varphi\|_{L^\infty(\O;\R^N)} \|v_\e\nabla u_\e\|_{L^2(\O;\R^N)} \|v_\e-1\|_{L^2(\O)} \to 0\,.$$
Gathering both information and using Proposition \ref{prop:lsc} leads to $(\eta_\e+v_\e^2)\nabla u_\e \wto \nabla u_*$ weakly* in $\mathcal D'(\O;\R^N)$. Since the sequence $\{(\eta_\e+v_\e^2)\nabla u_\e\}$ is bounded in $L^2(\O;\R^N)$, its weak $L^2$-convergence  follows. We can thus pass to the limit in \eqref{eq:varformu} in the sense of distribution and conclude that $\dive(\nabla u_*)=0$ in $\mathscr D'(\O)$. 

Since $\nabla u_*$  belongs to $L^2(\O;\R^N)$, and  $\dive(\nabla u_*)=0$, the normal trace $\nabla u_*\cdot\nu_\O$ is well defined as an element of $H^{-\frac12}(\partial\O)$. Recalling that $v_\e=1$ on $\partial\O$, we get that 
$$(\eta_\e+1)\partial_\nu u_\e =(\eta_\e+v_\e^2)\nabla u_\e \cdot \nu \rightharpoonup \nabla u_*\cdot \nu_\O \quad \text{ weakly in }H^{-\frac12}(\partial\O)\,.$$
We now improve this convergence into a weak convergence in $L^2(\partial\O)$. For that, let us consider a test function $X \in \C^1_c(\R^N;\R^n)$ such that $X=\nu$ on $\partial\O$ {in relation \eqref{eq:first-var}. Using that the left-hand side of \eqref{eq:first-var}} is clearly controlled by the Ambrosio-Tortorelli energy (see \eqref{eq:hypnrj}), we infer that 
$$\sup_{\e>0} \int_{\partial\O} \left[(\eta_\e+1)(\partial_\nu u_\e)^2+\e(\partial_\nu v_\e)^2\right]  \dd\HH^{N-1}<\infty\,.$$
On the one hand, we obtain that $\{\partial_\nu u_\e\}_{\e>0}$ is bounded in $L^2(\partial \O)$, hence $\partial_\nu u_\e \wto \nabla u_*\cdot \nu_\O$ weakly in $L^2(\partial\O)$.  On the other hand, there exists a nonnegative Radon measure $\lambda_* \in \M^+(\partial\O)$ such that $[(\partial_\nu u_\e)^2+\e(\partial_\nu v_\e)^2] \HH^{N-1}\res\partial\O \wsto \lambda_*$ weakly* in $\M(\partial\O)$.
\end{proof}

\begin{remark}\label{finalremark}
{\rm Our choice of Dirichlet boundary conditions for both $u$ and $v$ in \eqref{eq:PDE} allows one to obtain an $\e$-dependent boundary term which is nonnegative in the boundary integral involving $X\cdot\nu_\O$ in \eqref{eq:first-var}. This sign information is essential to get a limit boundary term which is a measure $\lambda_*$ concentrated on~$\partial\O$. If we had chosen a Neumann condition for $v$ and a Dirichlet condition for $u$ as in \cite{Francfort_Le_Serfaty_2009}, one would have obtained a more involved boundary term which would lead to a first order distribution on $\partial\O$ in the $\e \to 0$ limit.  It is not clear in this case how to perform the analysis in Section \ref{sec:convergence} (in particular Lemma \ref{lem:intpt}).}
\end{remark}

\section{Convergence of critical points}\label{sec:convergence}

Our objective is to show that ${u_*}$ is a critical point of the Mumford-Shah functional. We now improve the convergence results established at the previous section by additionally assuming the convergence of the energy \eqref{eq:hypothesis}, i.e., $AT_\e(u_\e,v_\e) \to MS(u_*)$. 
Under this stronger assumption, we can improve the above established convergences and in particular obtain the {equipartition} of the phase-field energy.

\begin{proposition}\label{prop:conv_abs_w_equi_repartition}
Let $\O \subset \R^N$ be a bounded open set with Lipschitz boundary and $g \in H^{\frac12}(\partial\O) \cap L^\infty(\partial\O)$. Let us consider a critical point  $(u_\e,v_\e)$ of the Ambrosio-Tortorelli functional satisfying the energy convergence \eqref{eq:hypothesis} and let $u_* \in SBV^2(\O)$ be given by Proposition \ref{prop:lsc}. Then, up to a further subsequence
\begin{equation}\label{eq:strong-conv}
\sqrt{\eta_\e+v^2_\e}\, \nabla u_\e \to \nabla u_*\,, \quad v_\e \nabla u_\e \to \nabla u_* \quad \text{ strongly in }L^2(\O;\R^N)\,.
\end{equation}
Moreover, setting $\Phi(t):=t-t^2/2$, and 
\begin{equation}\label{defweps}
w_\e:=\Phi(v_\e)\,,
\end{equation} 
then
\begin{equation}\label{eq;weps}
\begin{cases}
\nabla w_\e\LL^N \res\O \wsto 0 & \text{ weakly* in }\M(\O;\R^N)\,,\\
|\nabla w_\e| \LL^N \res\O \wsto  \HH^{N-1}\res \widehat J_{u_*} & \text{ weakly* in }\M(\overline\O)\,,
\end{cases}
\end{equation}
where we recall that $\widehat J_{u_*}=J_{u_*} \cup (\partial \O \cap \{u_*\neq g\})$. Finally, there is {equipartition} of the phase-field energy, i.e.,
\begin{equation}\label{eq:equi-rep}
\lim_{\e\to 0}\int_\O \left|\e |\nabla v_\e|^2 - \frac{1}{4\e}(1-v_\e)^2 \right|\dd x=\lim_{\e\to 0}\int_\O \left|2\e |\nabla v_\e|^2 -|\nabla w_\e| \right|\dd x=0\,.
\end{equation}
\end{proposition}

\begin{proof}
According to the convergence of energy assumption \eqref{eq:hypothesis} and the lower semicontinuity properties \eqref{eq:lsc-bulk}-\eqref{eq:lsc-phase_field} established in Proposition \ref{prop:lsc} (which applies by Lemma \ref{lem:max_princ_for_u}), we deduce that 
\begin{equation}\label{eq:1}
\int_\O |\nabla u_*|^2\dd x = \lim_{\e \to 0} \int_\O v_\e^2 |\nabla u_\e|^2\dd x= \lim_{\e \to 0} \int_\O(\eta_\e+ v_\e^2) |\nabla u_\e|^2\dd x
\end{equation}
and
\begin{equation}\label{eq:2}
  \HH^{N-1}({\widehat J_{u_*}}) = \lim_{\e \to 0} \int_{\O} (1-v_\e)|\nabla v_\e|\dd x=\lim_{\e \to 0} \int_\O\left(\e|\nabla v_\e|^2 +\frac{(v_\e-1)^2}{4\e}\right) \dd x\,.
\end{equation}
Convergence \eqref{eq:1} combined with the weak $L^2$-convergence of $v_\e\nabla u_\e$ to $\nabla u_*$ implies that $v_\e \nabla u_\e \to \nabla u_*$ strongly in $L^2(\O;\R^N)$. Moreover, it follows from \eqref{eq:1} that 
$\sqrt{\eta_\e}\, \nabla u_\e \to 0$ strongly in $L^2(\O;\R^N)$.   
Hence $\sqrt{\eta_\e+v^2_\e}\, \nabla u_\e \to \nabla u_*$ strongly in $L^2(\O;\R^N)$ as well.

Next, setting $w_\e=\Phi(v_\e)$, with \(\Phi(t)=t-t^2/2\), and using that $v_\e \to 1$ strongly in $L^2(\O)$ yields $w_\e \to 1/2$ strongly in $L^1(\O)$. Furthermore, owing to the chain rule in Sobolev spaces, we have $\nabla w_\e=\Phi'(v_\e)\nabla v_\e=(1-v_\e)\nabla v_\e$. In view of \eqref{eq:2}, we deduce that $\{\nabla w_\e\}_{\e>0}$ is bounded in $L^1(\O;\R^N)$, and thus $\nabla w_\e \LL^N\res\O \wsto 0$ weakly* in $\M(\O;\R^N)$. Moreover, localizing the conclusion of Proposition \ref{prop:lsc}, we get that for every open set $U \subset \O' \subset \R^N$, we have $\HH^{N-1}({\widehat J_{u_*}\cap U}) \leq \liminf_{\e} \int_U |\nabla w_\e|\dd x$, 
and, using \eqref{eq:2} it shows that
\begin{equation}\label{eq:3}
\int_\O |\nabla w_\e|\dd x \to \HH^{N-1}({\widehat J_{u_*}})  .
\end{equation}
Thus \cite[Theorem 1.80]{Ambrosio_Fusco_Pallara_2000} implies that $|\nabla w_\e|\LL^N\res\O\wsto  \HH^{N-1}\res \widehat J_{u_*}$ weakly* in $\M(\overline\O)$. 

The {equipartition} of energy is obtained by observing that
$$AT_\e(u_\e,v_\e)=\int_\O (\eta_\e+v_\e^2)|\nabla u_\e|^2\dd x + \int_\O |\nabla w_\e|\dd x +\int_\O \left|\sqrt\e |\nabla v_\e|-\frac{1}{2\sqrt\e}(1-v_\e) \right|^2\dd x\,,$$
and by using \eqref{eq:hypothesis}, \eqref{eq:1} and \eqref{eq:3}. Indeed, we get that 
$$\int_\O \left|\sqrt\e |\nabla v_\e|-\frac{1}{2\sqrt\e}(1-v_\e) \right|^2\dd x\to 0\,,$$
and then, by the Cauchy-Schwarz inequality,
\begin{eqnarray*}
\int_\O \left|\e |\nabla v_\e|^2 - \frac{1}{4\e}(1-v_\e)^2 \right|\dd x & \leq & \left\|\sqrt\e |\nabla v_\e|-\frac{1}{2\sqrt\e}(1-v_\e) \right\|_{L^2(\O)} \left\|\sqrt\e |\nabla v_\e|+\frac{1}{2\sqrt\e}(1-v_\e) \right\|_{L^2(\O)}\\
& \leq &C \left\|\sqrt\e |\nabla v_\e|-\frac{1}{2\sqrt\e}(1-v_\e) \right\|_{L^2(\O)} \to 0\,.
\end{eqnarray*}
Finally, using again that $|\nabla w_\e|=(1-v_\e)|\nabla v_\e|$, we observe that
\begin{eqnarray*}
\int_\O \left|2\e |\nabla v_\e|^2 -|\nabla w_\e| \right|\dd x& =& \int_\O \left| \e|\nabla v_\e|^2 -\frac{(1-v_\e)^2}{4\e} +\left(\sqrt\e |\nabla v_\e|-\frac{1}{2\sqrt\e}(1-v_\e)\right)^2 \right| \dd x\\
& \leq & \int_\O \left|\e |\nabla v_\e|^2 - \frac{(1-v_\e)^2}{4\e} \right|\dd x+\int_\O \left|\sqrt\e |\nabla v_\e|-\frac{1}{2\sqrt\e}(1-v_\e) \right|^2\dd x\to 0\,.
\end{eqnarray*}
This implies \eqref{eq:equi-rep}.
\end{proof}

The proof of Theorem \ref{thm:main} is based on (geometric) measure theoretic arguments. Let us define the $(N-1)$-varifold $V_\e \in \mathbf V_{N-1}(\overline\O)$ associated to the phase-field variable \(v_\e \in H^1(\O)\) by
$$\langle V_\e,\varphi\rangle:= \int_{\O\cap \{\nabla w_\e \neq 0\}}\varphi\left( x, {\rm Id}-\frac{\nabla w_\e}{|\nabla w_\e|} \otimes  \frac{\nabla w_\e}{|\nabla w_\e|} \right) |\nabla w_\e|  \dd x \quad \text{ for all }\varphi\in \C(\overline\O\times \mathbf G_{N-1})\,,$$ 
where \(w_\e:=\Phi(v_\e),\) and \(\Phi(t)=t-t^2/2\) for \(t \in [0,1]\). By the coarea formula, this definition is equivalent to the definition of a varifold associated to a function in \cite{Hutchinson_Tonegawa_2000}.  By standard compactness of bounded Radon measures, at the expense of extracting a further subsequence, there exists a varifold $V_*\in \mathbf V_{N-1}(\overline\O)$ such that $V_\e \wsto V_*\quad \text{ weakly* in }\M(\overline\O\times \mathbf G_{N-1})$.   

Note that $\|V_\e\| \wsto \|V_*\|$ weakly* in $\M(\overline\O)$ by definition of the mass of a varifold and thanks to the compactness of \( \mathbf G_{N-1}\). Recalling the definition  of $w_\e$ in \eqref{defweps}, we observe that $\|V_\e\|=|\nabla w_\e|\LL^N\res\O \wsto \HH^{N-1}\res \widehat J_{u_*}$  weakly* in $\M(\overline\O)$ according to \eqref{eq;weps}, and it follows  that $\|V_*\|= \HH^{N-1}\res \widehat J_{u_*}$. 
According to the disintegration Theorem (\cite[Theorem 2.28]{Ambrosio_Fusco_Pallara_2000}), there exists a weak* $\HH^{N-1}$-measurable mapping $x \in \overline\O \mapsto V_x \in \M(\mathbf G_{N-1})$ of probability measures such that $V_*=(\HH^{N-1}\res \widehat J_{u_*}) \otimes V_x $, i.e., for all $\varphi \in \C(\overline\O \times \mathbf G_{N-1})$,
\begin{equation}\label{eq:desint}
\int_{\overline\O \times \mathbf G_{N-1}} \varphi(x,A)\dd V_*(x,A)=\int_{\widehat J_{u_*}} \left(\int_{\mathbf G_{N-1}} \varphi(x,A)\dd V_x(A)\right)\dd\HH^{N-1}(x)\,.
\end{equation}
For $\HH^{N-1}$ almost every $x \in \widehat J_{u_*}$, we set 
\begin{equation}\label{eq:defAbar}
\overline A(x):=\int_{\mathbf G_{N-1}} A\dd V_x(A)\,.
\end{equation}

Owing to our various convergence results, we are now in position to pass to the limit in the inner variation equation \eqref{eq:first-var}. The limit expression is for now depending on the abstract limit varifold $V_*$ through its first moment $\overline A$ of $V_x$, and the abstract boundary measure $\lambda_*$ introduced in Lemma \ref{lem:conv-bdry}.

\begin{lemma}\label{lem:varifold}
Let $\O\subset \R^N$ is a bounded open set of class $\C^{2,1}$ and $g \in \C^{2,\alpha}(\partial\O)$  for some $\alpha \in (0,1)$. Let \(u_* \in SBV^2(\O)\) be a limit of critical points of the Ambrosio-Tortorelli functional as in Proposition~\ref{prop:conv_abs_w_equi_repartition}. For all $X \in \C^1_c(\R^N;\R^N)$, we have
\begin{multline}\label{eq:first-varter}
\int_\O \left( |\nabla u_*|^2 {\rm Id}-2\nabla u_* \otimes \nabla u_* \right):DX \dd x+\int_{\widehat J_{u_*}} \overline A:D X\dd \HH^{N-1} \\
= -\int_{\partial\O} (X\cdot\nu_\O) \dd\lambda_*+\int_{\partial\O}  |\nabla_\tau g|^2(X\cdot\nu_\O)\dd \HH^{N-1}-2\int_{\partial\O} (\nabla u_*\cdot \nu_\O) (X_\tau\cdot\nabla_\tau g)\dd\HH^{N-1}\,.
\end{multline}
\end{lemma}

\begin{proof}
Using the strong convergence \eqref{eq:strong-conv} established before, it is easy to pass to the limit in the first integral and in the left hand side of \eqref{eq:first-var}. We get for all $X \in \C^1_c(\R^N;\R^N)$, 
\begin{multline}\label{eq:inner1}
\int_\O \left(2 (\sqrt{\eta_\e+v^2_\e} \,\nabla u_\e) \otimes (\sqrt{\eta_\e+v^2_\e}\, \nabla u_\e)-(\eta_\e+v_\e^2)|\nabla u_\e|^2 {\rm Id} \right):D X\dd x\\
 \mathop{\longrightarrow}\limits_{\e\to0} \int_\O \left(2 \nabla u_* \otimes \nabla u_*-|\nabla u_*|^2 {\rm Id} \right):D X\dd x\,.
\end{multline}
According to Lemma \ref{lem:conv-bdry} we can also pass to the limit in the boundary integrals in the right hand side of \eqref{eq:first-var}, we get that
\begin{multline}\label{eq:inner1bis}
\int_{\partial\O} \left[(\eta_\e+1)(\partial_\nu u_\e)^2+\e(\partial_\nu v_\e)^2\right] X \cdot \nu \dd\HH^{N-1}\\
-(1+\eta_\e)\int_{\partial\O}X\cdot \nu |\nabla_\tau g|^2\dd\HH^{N-1}+2(\eta_\e+1)\int_{\partial \O} (\partial_\nu u_\e)(X_\tau\cdot\nabla_\tau g) \dd\HH^{N-1}\\
 \mathop{\longrightarrow}\limits_{\e\to0}  \int_{\partial\O} (X\cdot\nu_\O) \dd\lambda_*-\int_{\partial\O}  |\nabla_\tau g|^2(X\cdot\nu_\O)\dd \HH^{N-1}+2\int_{\partial\O} (\nabla u_*\cdot \nu_\O) (X_\tau\cdot\nabla_\tau g) \dd\HH^{N-1}\,.
\end{multline}

It remains to pass to the limit in the second integral in the left hand side of \eqref{eq:first-var}. Using the chain rule, we have $\nabla w_\e=\Phi'(v_\e)\nabla v_\e$. The {equipartition} of energy \eqref{eq:equi-rep} thus implies that
\begin{multline}\label{eq:inner2}
\lim_{\e \to 0}  \int_\O\left(2\e \nabla v_\e \otimes \nabla v_\e - \e|\nabla v_\e|^2{\rm Id} -\frac{(v_\e-1)^2}{4\e}{\rm Id}\right):DX  \dd x\\
=\lim_{\e \to 0}  \int_{\O \cap \{\nabla v_\e \neq 0\} }2\e|\nabla v_\e|^2 \left(\frac{\nabla v_\e}{|\nabla v_\e|} \otimes  \frac{\nabla v_\e}{|\nabla v_\e|}  - {\rm Id}\right):D X \dd x\\
= \lim_{\e \to 0}  \int_{\O \cap \{ \nabla w_\e \neq 0\}} |\nabla w_\e| \left( \frac{\nabla w_\e}{|\nabla w_\e|} \otimes  \frac{\nabla w_\e}{|\nabla w_\e|}  - {\rm Id}\right):D X \dd x\\
=-\int_{\overline\O \times \mathbf G_{N-1}} A:D X(x)\dd V_*(x,A)=-\int_{\widehat J_{u_*}} \overline A:D X \dd \HH^{N-1}.
\end{multline}
Gathering \eqref{eq:inner1}, \eqref{eq:inner1bis} and \eqref{eq:inner2}, we infer that \eqref{eq:first-varter} holds.
\end{proof}

Let us now identify the first moment $\overline A$ of the measure $V_x$. We first establish some algebraic properties of this matrix.

\begin{lemma}\label{lem:barA}
For $\HH^{N-1}$-almost every $x \in \widehat J_{u_*}$, the matrix $\overline A(x)$ satisfies
$$\overline A(x) \geq 0, \quad {\rm tr}(\overline A(x))=N-1, \quad \rho(\overline A(x))=1,$$
where $\rho$ denotes the spectral radius.
\end{lemma}

\begin{proof}
To simplify notation,  we set
$$A_\e:=\left({\rm Id} -\frac{\nabla w_\e}{|\nabla w_\e|} \otimes  \frac{\nabla w_\e}{|\nabla w_\e|}\right)\,.$$
The matrix \(A_\e\) is well-defined on the set \( \O \cap \{ \nabla w_\e \neq 0 \}\), it is a symmetric matrix corresponding to the orthogonal projection on $\{\nabla w_\e\}^\perp$.
It satisfies $A_\e \geq 0$,  ${\rm tr}(A_\e)=N-1$, and  $\rho(A_\e)=1$ in $\O\cap \{ \nabla w_\e \neq 0 \}$. 
For all $\varphi \in \C^0(\overline\O)$, we have
\begin{multline*}
\int_{\widehat J_{u_*}} {\rm tr}(\overline A) \varphi\dd\HH^{N-1}  =  \int_{\overline\O \times \mathbf G_{N-1}} {\rm tr}(A) \varphi(x)\dd V_*(x,A) =  \lim_{\e \to 0}\int_{\O \times \mathbf G_{N-1}} {\rm tr}(A) \varphi(x)\dd V_\e(x,A)\\
 =  \lim_{\e\to 0}\int_{\O \cap \{ \nabla w_\e \neq 0\}}  {\rm tr}(A_\e)\varphi |\nabla w_\e|  \dd x 
 =  (N-1)\lim_{\e \to 0} \int_\O\varphi |\nabla w_\e|\dd x 
 =  (N-1)\int_{\widehat J_{u_*}}\varphi\dd\HH^{N-1}\,,
\end{multline*}
which shows that ${\rm tr}(\overline A)=(N-1)$ $\HH^{N-1}$-a.e.\ on $\widehat J_{u_*}$.  If further $\varphi \geq 0$ and $z \in \R^N$, then
\begin{multline*}
\int_{\widehat J_{u_*}}  (\overline Az\cdot z)\varphi\dd\HH^{N-1}  =  \int_{\overline\O \times \mathbf G_{N-1}} (Az\cdot z) \varphi(x)\dd V_*(x,A)\\
 =  \lim_{\e \to 0}\int_{\O \times \mathbf G_{N-1}}(Az\cdot z) \varphi(x)\dd V_\e(x,A)
 =  \lim_{\e\to 0}\int_{\O \cap \{ \nabla w_\e \neq 0\} } \left(A_\e z\cdot z\right)\varphi |\nabla w_\e|  \dd x\geq 0\,.
\end{multline*}
As a consequence, for all $z \in \R^N$, we have $\overline A z\cdot z \geq 0$ $\HH^{N-1}$-a.e.\ on $\widehat J_{u_*}$, from which we deduce that $\overline A$ is a nonnegative matrix $\HH^{N-1}$-a.e.\ on $\widehat J_{u_*}$. 

Since for all $\varphi \in \C^0(\overline\O)$ we have
$$\int_{\O \cap \{ \nabla w_\e \neq 0\}} |\nabla w_\e| \left({\rm Id}-\frac{\nabla w_\e}{|\nabla w_\e|} \otimes  \frac{\nabla w_\e}{|\nabla w_\e|} \right)\varphi \dd x\to \int_{\overline\O \times \mathbf G_{N-1}} A\varphi(x)\dd V_*(x,A)=\int_{\widehat J_{u_*}} \overline A\varphi \dd \HH^{N-1}\,,$$
we deduce that
\begin{equation}\label{eq:crucial}
A_\e|\nabla w_\e|\LL^N\res\O \wsto \overline A \,  \HH^{N-1}\res \widehat J_{u_*} \quad \text{ weakly* in }\M(\overline\O;\mathbf M^{N \times N})\,.
\end{equation}
Using that the spectral radius $\rho$ is a convex, continuous, and positively $1$-homogeneous function on the set of symmetric matrices, it follows from Reshetnyak continuity Theorem (see \cite[Theorem 2.39]{Ambrosio_Fusco_Pallara_2000}) that for all $\varphi \in \C^0(\overline\O)$,
$$\int_{\widehat J_{u_*}} \varphi \, \dd\HH^{N-1}=\lim_{\e \to 0}\int_\O \varphi |\nabla w_\e|\dd x=\lim_{\e \to 0}\int_\O \varphi\rho(A_\e)|\nabla w_\e|\dd x = \int_{\widehat J_{u_*}} \varphi \rho(\overline A) \,\dd\HH^{N-1}\,,$$
hence $\rho(\overline A)=1$ $\HH^{N-1}$-a.e.\ on $\widehat J_{u_*}$. 
\end{proof}

We now focus on the interior structure of the varifold $V_*$.

\begin{lemma}\label{lem:intpt}
For $\HH^{N-1}$-a.e.\ \(x\) in $J_{u_*}=\widehat J_{u_*} \cap \O$, we have $\overline A(x)= {\rm Id}-\nu_{u_*}(x) \otimes \nu_{u_*}(x)$, 
where $\nu_{u_*}$ is the approximate normal to the countably $\HH^{N-1}$-rectifiable set $J_{u_*}$.
\end{lemma}

\begin{proof}

{\it Step 1: Let us show that for $\HH^{N-1}$-a.e.\ $x$ in  $J_{u_*}$, $\overline A(x)$ is a projection matrix onto a $(N-1)$-dimensional hyperplane.} 
To this aim, we perform a blow-up argument on the first variation equation \eqref{eq:first-varter}. Let $x_0 \in J_{u_*}$ be such that
\begin{enumerate}
\item $x_0$ is a Lebesgue point of $\overline A$ with respect to $\HH^{N-1}\res J_{u_*}$;
\item $J_{u_*}$ admits an approximate tangent space at $x_0$ given by $T_{x_0}=\{\nu_{u_*}(x_0)\}^\perp$;
\item
\begin{equation}\label{eq:BMM}
\lim_{\varrho\to 0}\frac{\HH^{N-1}(J_{u_*} \cap B_\varrho(x_0))}{\omega_{N-1}\varrho^{N-1}}=1;
\end{equation}
\item
$$\lim_{\varrho\to 0}\frac{1}{\varrho^{N-1}}\int_{B_\varrho(x_0)}|\nabla u_*|^2\dd x=0\,.$$
\end{enumerate}

It turns out that $\HH^{N-1}$-almost every point $x_0  \in J_{u_*}$ satisfies these properties. Indeed, (1) is a consequence of the Besicovitch differentiation theorem, (2) and (3) are consequences of the rectifiability of $ J_{u_*}$ (see Theorems 2.63 and 2.83 in \cite{Ambrosio_Fusco_Pallara_2000}), while condition (4) is a consequence of (3)  together with the fact that the measure $|\nabla u_*|^2\LL^N\res\O$ is singular with respect to $\HH^{N-1}\res J_{u_*}$.

Let $x_0 \in J_{u_*}$ be such a point and let $\varrho>0$ be such that $\overline{B_\varrho(x_0)}\subset \O$. For $\phi \in \C^\infty_c(\R^N;\R^N)$  such that ${\rm Supp}(\phi)\subset B_1$, we set $\phi_{x_0,\varrho}(x):=\phi\left(\frac{x-x_0}{\varrho}\right)$ for  $x \in \R^N$\,, 
so that $\phi_{x_0,\varrho} \in \C^\infty_c(\R^N;\R^N)$ and ${\rm Supp}(\phi_{x_0,\varrho}) \subset B_\varrho(x_0)$. Taking $\phi_{x_0,\varrho}$ as test vector field in \eqref{eq:first-varter} (note that $\phi_{x_0,\varrho}=0$ in a neighbourhood of $\partial\O$) yields
$$\int_{J_{u_*} \cap B_\varrho(x_0)} \overline A:D\phi_{x_0,\varrho}\dd\HH^{N-1}=-\int_{B_\varrho(x_0)} \left( |\nabla u_*|^2 {\rm Id} -2\nabla u_* \otimes \nabla u_*\right):D\phi_{x_0,\varrho}\dd x\,.$$
Dividing this identity by $\varrho^{N-2}$ yields 
\begin{multline*}
\frac{1}{\varrho^{N-1}}\int_{ J_{u_*} \cap B_\varrho(x_0)} \overline A:D\phi\left(\frac{\cdot-x_0}{\varrho}\right)\dd\HH^{N-1}\\
=\frac{-1}{\varrho^{N-1}}\int_{B_\varrho(x_0)} \left(|\nabla u_*|^2 {\rm Id}-2 \nabla u_* \otimes \nabla u_* \right):D \phi\left(\frac{\cdot-x_0}{\varrho}\right)\dd x\,.
\end{multline*}

We first show that the left hand side of the previous equality tends to zero as $\varrho \to 0$. Indeed, thanks to our choice of $x_0$, we have
\begin{multline*}
\left|\frac{1}{\varrho^{N-1}}\int_{B_\varrho(x_0)} \left(|\nabla u_*|^2 {\rm Id}- 2\nabla u_* \otimes \nabla u_* \right):D \phi\left(\frac{x-x_0}{\varrho}\right)\dd x\right|\\
\leq C \frac{\|D\phi\|_{L^\infty(B_1;\mathbf M^{N \times N})}}{\varrho^{N-1}}\int_{B_\varrho(x_0)}|\nabla u_*|^2\dd x \to 0\,,
\end{multline*}
for some constant \(C>0\).
For what concerns the right hand side, using first that $x_0$ is a Lebesgue point of $\overline A$ and \eqref{eq:BMM}, we get that
\begin{multline*}
\left|\frac{1}{\varrho^{N-1}}\int_{J_{u_*} \cap B_\varrho(x_0)} (\overline A-\overline A(x_0)):D\phi\left(\frac{\cdot-x_0}{\varrho}\right) \dd\HH^{N-1}\right|\\
\leq \frac{\|D\phi\|_{L^\infty(B_1;\mathbf M^{N \times N})}}{\varrho^{N-1}} \int_{J_{u_*} \cap B_\varrho(x_0)} |\overline A-\overline A(x_0)|\dd\HH^{N-1}\to 0\,,
\end{multline*}
so that
\begin{multline*}
\lim_{\varrho\to 0}\frac{1}{\varrho^{N-1}}\int_{J_{u_*} \cap B_\varrho(x_0)} \overline A:D\phi\left(\frac{\cdot-x_0}{\varrho}\right)\dd\HH^{N-1}\\
= \overline A(x_0):\lim_{\varrho\to 0}\frac{1}{\varrho^{N-1}}\int_{J_{u_*} \cap B_\varrho(x_0)}D\phi\left(\frac{\cdot-x_0}{\varrho}\right)\dd\HH^{N-1}\,.
\end{multline*}
Using next that $J_{u_*}$ admits an approximate tangent space that we denote by $T_{x_0}$ at $x_0$, we obtain that
$$\lim_{\varrho\to 0}\frac{1}{\varrho^{N-1}}\int_{J_{u_*} \cap B_\varrho(x_0)}D\phi\left(\frac{\cdot-x_0}{\varrho}\right)\dd\HH^{N-1}=\int_{T_{x_0}\cap B_1} D\phi\dd\HH^{N-1}\,.$$
Hence, 
\begin{equation}\label{eq:stationnary}
\int_{T_{x_0}\cap B_1} \overline A(x_0):D\phi\dd\HH^{N-1}=0 \quad \text{ for all }\phi \in \C^\infty_c(B_1;\R^N)\,.
\end{equation}

Let $t \in (0,1)$ be such that $t<\sqrt{N}^{-N}$, the measure $\nu:=\frac{1}{\omega_{N-1}}\HH^{N-1}\res T_{x_0}$ satisfies 
$$t^{N-1} \leq \nu(\overline B_t) \leq \nu(B_1) \leq 1.$$
According to \cite[Lemma 3.9]{Ambrosio_Soner_1997} with $\beta=s=N-1$, we get that the matrix $\overline A(x_0)$ has at most $N-1$ nonzero eigenvalues. Recalling that ${\rm tr}(\overline A(x_0))=N-1$ and that all eigenvalues of $\overline A(x_0)$ belong to $[0,1]$, this implies that $\overline A(x_0)$ has exactly $N-1$ eigenvalues which are equal to $1$, and one eigenvalue which is zero. Hence, there exists $e \in \mathbf S^{N-1}$ such that $A={\rm Id}-e\otimes e$.
\vskip3pt

\noindent{\it Step 2: Let us show that $e=\pm\nu_{u_*}(x_0)$.} Let us consider the varifold 
$$\widetilde V:=  \HH^{N-1}\res T_{x_0} \otimes \delta_{\overline A(x_0)} \in \M(B_1 \times \mathbf G_{N-1})\,,$$
whose action is given by
$$\int_{B_1 \times \mathbf G_{N-1}} \varphi(x,A)\dd \widetilde V(x,A)=\int_{B_1 \cap T_{x_0}} \varphi(x,\overline A(x_0))\dd\HH^{N-1}(x) \quad \text{ for all }\varphi \in \C^0_c(B_1 \times \mathbf G_{N-1})\,.$$
Since $\overline A(x_0)$ is a projection matrix onto the hyperplane $e^\perp$, it follows that $\widetilde V \in \mathbf V_{N-1}(B_1)$ is a $(N-1)$-varifold in $B_1$ with $\|\widetilde V\|=\HH^{N-1}\res T_{x_0}$. Moreover, condition \eqref{eq:stationnary} shows that $\widetilde V$ is a stationary varifold, cf. Section \ref{sec:varifold}. It follows from the monotonicity formula (see e.g.\ formula (40.3) page 236 in \cite{Simon_1983}) that for all $x \in T_{x_0} \cap B_1$ and all $\varrho>0$ such that $B_\varrho(x) \subset B_1$,
$$
\frac{\HH^{N-1}(T_{x_0} \cap B_\varrho(x))}{\varrho^{N-1}}= \frac{\HH^{N-1}(T_{x_0} \cap B_r(x))}{r^{N-1}} + 
 \int_{T_{x_0} \cap B_\varrho(x) \setminus B_r(x)}\frac{|e\cdot (y-x)|^2}{|y-x|^{N+1}}\dd\HH^{N-1}(y)$$
for all $0<r<\varrho$. Since
$$\frac{\HH^{N-1}(T_{x_0} \cap B_r(x))}{r^{N-1}}=\frac{\HH^{N-1}(T_{x_0} \cap B_\varrho(x))}{\varrho^{N-1}}=\omega_{N-1} \,,$$
we deduce that 
$$\int_{T_{x_0}\cap B_\varrho(x) \setminus B_r(x)} \frac{|e\cdot (y-x)|^2}{|y-x|^{N+1}}\dd\HH^{N-1}(y)=0\,.$$
Choosing $x=0$, $\varrho=1$, and letting $r \to 0^+$, we infer that $y\cdot e=0$ for $\HH^{N-1}$-a.e. $y \in T_{x_0}\cap B_1$ which implies that $T_{x_0}=e^\perp$, hence $e=\pm \nu_{u_*}(x_0)$.
\end{proof}

Next we focus on boundary points.

\begin{lemma}\label{lem:bdry}
For $\HH^{N-1}$-a.e. $x \in \widehat J_{u_*} \cap \partial\O$, we have $\overline A(x)={\rm Id} -\nu_\O(x) \otimes \nu_\O(x)$, 
where $\nu_\O$ is the outward unit normal to $\partial \O$.
\end{lemma}

\begin{proof}
We perform again a blow-up argument, this time at boundary points. Let $x_0 \in \widehat J_{u_*} \cap \partial\O$ be such that:
\begin{enumerate}
\item $x_0$ is a Lebesgue point of  $\overline A$ with respect to $\HH^{N-1}\res\widehat J_{u_*}$;
\item $\widehat J_{u_*}$ admits an approximate tangent space at $x_0$ which coincides with the (usual) tangent space of $\partial\O$ at $x_0$ (this in particular implies that $\nu_{u_*}(x_0)=\pm \nu_\O(x_0)$);
\item
$$\lim_{\varrho\to 0}\frac{\HH^{N-1}(\widehat J_{u_*} \cap B_\varrho(x_0))}{\omega_{N-1}\varrho^{N-1}}=1\,;$$
\item
$$\lim_{\varrho\to 0}\frac{1}{\varrho^{N-1}}\int_{B_\varrho(x_0)\cap \O}|\nabla u_*|^2\dd x=0\,;$$
\item
$$\lim_{\varrho \to 0}\frac{\lambda_*(B_\varrho(x_0))}{\varrho^{N-2}}=0\,, \qquad \lim_{\varrho \to 0}\frac{1}{\varrho^{N-2}}\int_{\partial\O \cap B_\varrho(x_0)} |\nabla u_*\cdot \nu|\dd\HH^{N-1}=0\,.$$
\end{enumerate}
It turns out that $\HH^{N-1}$ almost every point $x_0  \in \widehat J_{u_*} \cap \partial\O$ satisfies these properties. Indeed, (1) is a consequence of the Besicovitch differentiation Theorem while (2) comes from the rectifiability of $\widehat J_{u_*}$ (see \cite[Theorems 2.83]{Ambrosio_Fusco_Pallara_2000}) together with the locality of approximate tangent spaces (see \cite[Proposition 2.85]{Ambrosio_Fusco_Pallara_2000}). Condition (3) is again a consequence of the rectifiabilty of $\widehat J_{u_*}$ and the Besicovitch-Marstrand-Mattila theorem (see \cite[Proposition 2.63]{Ambrosio_Fusco_Pallara_2000}). Condition (4) is a consequence of (3)  together with the fact that the measure $|\nabla u_*|^2\LL^N\res\O$ is singular with respect to $\HH^{N-1}\res \widehat J_{u_*}$. To justify (5), we define for  $x \in \partial\O$,
$$\Theta(x):=\limsup_{\varrho\to 0}\frac{\lambda_*(B_\varrho(x))}{\varrho^{N-2}}\,.$$
According to \cite[Theorem 2.56]{Ambrosio_Fusco_Pallara_2000}, we have   
$t \HH^{N-2}(\{\Theta\geq t\}) \leq \lambda_*(\{\Theta\geq t\})<\infty$ for all $t>0$. 
Hence $\HH^{N-1}(\{\Theta\geq t\})=0$ for all $t>0$. As a consequence, $\HH^{N-1}(\{\Theta>0\})=0$. The second property of (5) can be obtained similarly replacing $\lambda_*$ by $|\nabla u_*\cdot \nu|\HH^{N-1}\res \partial\O$.

We choose  such a point $x_0 \in \partial \O \cap \widehat J_{u_*}$ and we take $\varrho>0$. 
\vskip3pt

\noindent {\it Step 1: We first prove that $\nu_\Omega(x_0)$ is an eigenvector of $\overline A(x_0)$.}  Consider first a test vector field $\phi$ of the form
$\phi(x):=\varphi\left(\frac{x-x_0}{\varrho}\right) \tilde \tau(x)$ for $x \in B_\varrho(x_0)$,  
where $\varphi\in \C_c^\infty(B_1)$ and $\tilde \tau \in \C_c^1(\R^N;\R^N)$ is such that $\tilde \tau\cdot\nu_\O=0$ on $\partial\O$. Plugging $\phi$ in \eqref{eq:first-varter} and using estimates similar to the proof of Lemma \ref{lem:intpt}, we obtain 
\begin{equation}\label{eq:stationnarybis}
\int_{T_{x_0}\cap B_1} \overline A(x_0):(\tilde \tau(x_0) \otimes \nabla \varphi)\dd\HH^{N-1}=0 \quad \text{ for all }\varphi \in \C^\infty_c(B_1)\,.
\end{equation}
Note that to get \eqref{eq:stationnarybis}, the boundary term is cancelled thanks to the second property of (5). Let $\{\tau_1,\ldots,\tau_{N-1}\}$ be an orthonormal basis of $T_{x_0}$, and $\nu:=\nu_\O(x_0)$ be the outward unit normal to $\O$ at $x_0$ (i.e. $\nu$ is a normal vector to $T_{x_0}$). We choose the vector field $\tilde \tau$ in such a way that $\tilde \tau(x_0)=\tau_i$,  and we decompose $\nabla\varphi$ along the orthonormal basis $\{\tau_1,\ldots,\tau_{N-1},\nu\}$ of $\R^N$ as $\nabla \varphi=\sum_{j=1}^{N-1} (\partial_j \varphi) \tau_j+(\partial_\nu \varphi)\nu$. 
Since $\int_{T_{x_0}\cap B_1} \partial_j \varphi\dd\HH^{N-1}=0$ for all $1 \leq j \leq N-1$, we infer from \eqref{eq:stationnarybis} that 
$$\Big((\overline A(x_0)\tau_i)\cdot\nu\Big) \int_{T_{x_0}\cap B_1} \partial_\nu \varphi\dd\HH^{N-1}=0\,.$$
From the arbitrariness of $\varphi$, it follows that $(\overline A(x_0)\tau_i)\cdot\nu=0$ for all $1\leq i \leq N-1$. Since $\overline A(x_0)$ is symmetric, we deduce that $\overline A(x_0)\nu\in T_{x_0}^\perp$, that is 
$\overline A(x_0)\nu=c\nu$ for some $c \in [0,1]$ (recall that all eigenvalues of $\overline A(x_0)$ belong to $[0,1]$ by Lemma~\ref{lem:barA}). Thus $\nu$ is an eigenvector of $\overline A(x_0)$, and by the spectral theorem, we can also assume without loss of generality that $\tau_1,\ldots,\tau_{N-1}$ are also eigenvectors of $\overline A(x_0)$. 
\vskip3pt

\noindent{\it Step 2:  We next show that $\overline A(x_0)$ is the projection matrix onto the tangent space to $\partial\O$ at $x_0$.}

 We now consider a test vector field $\phi$ of the form $\phi(x):=\tilde \nu(x)\varphi\left(\frac{x-x_0}{\varrho}\right)$ for $x \in B_\varrho(x_0)$,  
where $\varphi\in \C_c^\infty(B_1)$ and $\tilde \nu \in \C_c^1(\R^N;\R^N)$ is such that $\tilde \nu(x)$ is a normal vector to $\partial\O$ at $x \in \partial\O$ satisfying $\tilde \nu(x_0)=\nu_\O(x_0)$. Using again estimates in a similar way to the proof of Lemma \ref{lem:intpt} (this time, the boundary term is cancelled thanks to the first property of (5)), we obtain that
$$\int_{T_{x_0}\cap B_1} \overline A(x_0):(\nu \otimes \nabla \varphi)\dd\HH^{N-1}=0 \quad \text{ for all }\varphi \in \C^\infty_c(B_1)\,,$$
and thus, by Step 1, $c \int_{T_{x_0}\cap B_1} \partial_{\nu} \varphi\dd\HH^{N-1}=0$. 
By arbitrariness of $\varphi$, this last equality shows that $c=0$. As a consequence, there exist real numbers $c_1,\ldots,c_{N-1} \in [0,1]$ (the eigenvalues of $\overline A(x_0)$ associated to the eigenvectors $\tau_1,\ldots,\tau_{N-1}$) such that $\overline A(x_0)=\sum_{i=1}^{N-1} c_i \tau_i \otimes \tau_i$. 
According to Lemma~\ref{lem:barA}, ${\rm tr}(\overline A(x_0))=c_1+\cdots +c_{N-1}=N-1$,  and we deduce that $c_1=\cdots=c_{N-1}=1$. Hence, $\overline A(x_0)=\sum_{i=1}^{N-1} \tau_i \otimes \tau_i={\rm Id} -\nu \otimes \nu$ 
as announced.
\end{proof}

We can now complete the proof of our second main theorem.

\begin{proof}[Proof of Theorem \ref{thm:main}]

i) This point is a consequence of Lemma \ref{lem:conv-bdry}.

ii) Using that $\nu_{u_*}=\pm \nu_\O$ $\HH^{N-1}$-a.e. in $\partial \O \cap \widehat J_{u_*}$ and gathering Lemmas \ref{lem:intpt} and \ref{lem:bdry} yields
$\overline A={\rm Id}-\nu_{u_*} \otimes \nu_{u_*} \quad \HH^{N-1}\text{-a.e. in }\widehat J_{u_*}$. 
Thus, according to \eqref{eq:desint} and \eqref{eq:defAbar}, we get that
$$\int_{\widehat J_{u_*}} \overline A : D X\, \dd\HH^{N-1}=\int_{\widehat J_{u_*}} ({\rm Id}-\nu_{u_*}\otimes \nu_{u_*}) : D X\, \dd\HH^{N-1}.$$
Then Lemma \ref{lem:varifold} implies that 
\begin{multline*}
\int_\O \left( |\nabla u_*|^2 {\rm Id}-2\nabla u_*\otimes \nabla u_* \right):DX\dd x+\int_{\widehat J_{u_*}} ({\rm Id}-\nu_{u_*} \otimes \nu_{u_*}):D X\, \dd\HH^{N-1}\\
= -\int_{\partial\O} (X\cdot\nu_\O) \dd\lambda_*+\int_{\partial\O} |\nabla_\tau g|^2(X\cdot\nu_\O) \dd \HH^{N-1}-2\int_{\partial\O} (\nabla u_*\cdot \nu_\O) (X_\tau\cdot\nabla_\tau g)\dd\HH^{N-1}
\end{multline*}
for all $X \in \C^1_c(\R^N;\R^N)$. Specifying this identity  to  vector fields $X \in \C^1_c(\R^N;\R^N)$ satisfying $X\cdot\nu_\O=0$ on $\partial\O$ leads to
\begin{multline*}
\int_\O \left( |\nabla u_*|^2 {\rm Id}-2\nabla u_*\otimes \nabla u_* \right):DX\dd x+\int_{\widehat J_{u_*}} ({\rm Id}-\nu_{u_*} \otimes \nu_{u_*}):D X\, \dd\HH^{N-1}\\
=-2\int_{\partial\O} (\nabla u_*\cdot \nu_\O) (X\cdot\nabla_\tau g)\dd\HH^{N-1}\,,
\end{multline*} 
and \eqref{eq:first-var3} follows from the definition of the tangential divergence of $X$ on the countably $\HH^{N-1}$-rectifiable set $\widehat J_{u_*}$.
\end{proof}

The results of this section also give the following convergences that will be used in Section \ref{sec:passing_limit_2nd_variation}.

\begin{corollary}\label{prop:measures_conv}
Let \( (u_\e,v_\e)\in \mathcal A_g(\O) \),  $w_\e$ given by \eqref{defweps}, and $u_* \in SBV^2(\O) \cap L^\infty(\O)$ be as in Theorem~\ref{thm:main}. Then,
\begin{equation}\label{eq:measures_conv_1}
\frac{\nabla w_\e}{|\nabla w_\e|} \otimes \frac{\nabla w_\e}{|\nabla w_\e|} |\nabla w_\e| \LL^N\res\O \wsto \nu_{u_*} \otimes \nu_{u_*}   \mathcal{H}^{N-1}\res \widehat{J} _{u_*}  \quad \text{ weakly* in } \mathcal{M}(\overline{\O};\mathbf M^{N \times N})\,,
\end{equation}
\begin{equation}\label{eq:measures_conv_2}
\e \nabla v_\e \otimes \nabla v_\e \LL^N\res\O \wsto \frac{1}{2}\nu_{u_*} \otimes \nu_{u_*} \mathcal{H}^{N-1}\res \widehat{J} _{u_*}  \quad \text{ weakly* in } \mathcal{M}(\overline{\O};\mathbf M^{N \times N})\,. 
\end{equation}
\end{corollary}

\begin{proof}
The first point \eqref{eq:measures_conv_1} follows from \eqref{eq:crucial} together with Lemmas \ref{lem:intpt} and \ref{lem:bdry} by observing that  $\nu_{u_*}=\pm \nu_\O$ $\HH^{N-1}$-a.e. in $\partial \O \cap \widehat J_{u_*}$. The second point \eqref{eq:measures_conv_2} follows from the first one, the {equipartition} of energy \eqref{eq:equi-rep} and the fact that $\frac{\nabla v_\e}{|\nabla v_\e|}=\frac{\nabla w_\e}{|\nabla w_\e|}$.
\end{proof}

\section{Passing to the limit in the second inner variation}\label{sec:passing_limit_2nd_variation}\label{sec:inner_stability}

The aim of this section is to complement Theorem \ref{thm:main} proving also the convergence of the second inner variation of $AT_\e$. As a consequence, we shall deduce that if the limit $u_*$ comes from stable critical points of $AT_\e$, then $u_*$ satisfies a certain stability condition for $MS$.  
Our analysis and result parallels completely the ones in \cite{Le_2011,Le_2015,Le_Sternberg_2019} for the Allen-Cahn type energies arising in phase transitions problems. 

\begin{proof}[Proof of Theorem \ref{thm:main_2}]
Assume that $\partial\O$ is of class $\C^{3,1}$ and $g \in \C^{3,\alpha}(\partial\O)$ for some $\alpha \in (0,1)$. By Theorem \ref{thembdregloc}, if $(u_\e,v_\e)$ is a critical points of the Ambrosio-Tortorelli functional then it belongs to $[\C^{3,\alpha}(\overline\O)]^2$. 
To prove the convergence of the second inner variation, we  use Lemma \ref{lem:expressions_inner} and formula  \eqref{eq:2nd_inner_variation}. From Proposition \ref{prop:conv_abs_w_equi_repartition}, we know that 
$$\begin{cases}
\sqrt{\eta_\e+v_\e^2}\nabla u_\e \rightarrow \nabla u_* & \text{ strongly in }L^2(\O;\R^N)\,,\\[5pt]
\displaystyle \e |\nabla v_\e|^2 \LL^N\res\O \wsto \frac12 \mathcal{H}^{N-1}\res \widehat J_{u_*} & \text{ weakly* in }\M(\overline\O)\,,\\[5pt]
\displaystyle \frac{(v_\e-1)^2}{4\e} \LL^N\res\O \wsto \frac12 \mathcal{H}^{N-1}\res \widehat J_{u_*} & \text{ weakly* in }\mathcal{M}(\overline \O)\,.
\end{cases}$$
On the other hand, Corollary \ref{prop:measures_conv} ensures that
$\e\nabla v_\e \otimes \nabla v_\e \LL^N \res\O \wsto \frac12\nu_{u_*} \otimes \nu_{u_*} \mathcal{H}^{N-1}\res \widehat J_{u_*}$ weakly* in  $\M(\overline\O;\mathbf M^{N \times N})$. 
Let $X \in \C^2_c(\R^N;\R^N)$ and $G \in \C^2(\R^N)$ be such that $X\cdot \nu_\O=0$ and $G=g$ on $\partial\O$, and set $Y:=(DX)X$. Observing that \( |D X^T\nabla v_\e|^2=(DX (D X)^T) :(\nabla v_\e \otimes \nabla v_\e)\), we can pass to the limit in all the terms of \(\delta^2AT_\e(u_\e,v_\e)[X,G]\) in \eqref{eq:2nd_inner_variation} to find that
\begin{eqnarray}\label{conv-var2}
\lim_{\e \to 0} \delta^2 AT_\e(u_\e,v_\e)[X,G] 
 & = & \int_\O \left( |\nabla u_*|^2{\rm Id}-2(\nabla u_*\otimes \nabla u_*) \right):DY \dd x +\int_{\widehat{J}_{u_*}} \dive^{\widehat{J}_{u_*}}Y \dd \HH^{N-1} \nonumber \\
&& +\int_\O |\nabla u_*|^2\left( (\dive X)^2-\tr (DX)^2 \right) -4\big((\nabla u_*\otimes \nabla u_*) :DX\big) \dive X \dd x \nonumber \\
&&+\int_\O \left[4(\nabla u_*\otimes \nabla u_*): (DX)^2+2 |DX^T\nabla u_*|^2\right] \dd x \nonumber \\
&&+4 \int_\O \Big[\nabla u_*\cdot \nabla (X\cdot \nabla G) \dive X \nonumber\\
&&\hspace{3cm}-\big(\nabla u_*\otimes \nabla (X\cdot \nabla G)\big) :\big(DX+(DX)^T\big)\Big] \dd x\nonumber\\
&&+2 \int_\O \nabla u_*\cdot \nabla (X\cdot \nabla (X\cdot G)) \dd x +2\int_\O |\nabla (X\cdot \nabla G)|^2 \dd x\nonumber\\
&&+\int_{\widehat{J}_{u_*}} \left[(\dive X)^2-\tr (DX)^2  -2 \big((\nu_{u_*} \otimes \nu_{u_*}):DX\big) \dive X \right]\dd \HH^{N-1} \nonumber \\
&&+2\int_{\widehat{J}_{u_*}} \left[(\nu_{u_*} \otimes \nu_{u_*}): (DX)^2+|D X^T \nu_{u_*}|^2 \right]\dd \HH^{N-1}\,.
\end{eqnarray}
Using the geometric formulas stated in the proof of Theorem 1.1 p.\ 1851--1852 in \cite{Le_2011}, we infer that
\begin{multline}\label{eq:computation_div_proj}
(\dive X)^2-\tr[(DX)^2]-2 \big((\nu_{u_*} \otimes \nu_{u_*}):D X\big) \dive X+2 (\nu_{u_*} \otimes \nu_{u_*}): (D X)^2+|D X^T \nu_{u_*}|^2 \\
= (\dive^{J_{u_*}} X)^2+\sum_{i=1}^{N-1}|(\partial_{\tau_i}X)^\perp|^2- \sum_{i,j=1}^{N-1} \left( \tau_i\cdot \partial_{\tau_j} X \right) \left( \tau_j \cdot \partial_{\tau_i} X \right) +\left( (\nu_{u_*} \otimes \nu_{u_*}):DX\right)^2\,.
\end{multline}
According to the expression of the inner second variation of the Mumford-Shah energy stated in Lemma~\ref{lem:explicit_inner_stab2}, \eqref{conv-var2} and \eqref{eq:computation_div_proj}, we infer that
$$\lim_{\e \to 0} \delta^2 AT_\e(u_\e,v_\e)[X,G] =\delta^2 MS(u_*)[X,G]+\int_{\widehat J_{u_*}} \big( (\nu_{u_*} \otimes \nu_{u_*}):DX \big)^2\dd \HH^{N-1}\,.$$

Now assume  that \((u_\e,v_\e)\in \mathcal A_g(\O)\) is a stable critical point of \(AT_\e\), i.e.
\begin{equation}\label{stabinsect}
\dd^2 AT_\e(u_\e,v_\e)[\phi,\psi] \geq 0 \quad \text{ for all } (\phi ,\psi) \in H^1_0(\O)\times[H^1_0(\O)\cap L^\infty(\O)]\,,
\end{equation}
where $\dd^2 AT_\e(u_\e,v_\e)$ is the second outer variation of $AT_\e$ at $(u_\e,v_\e)$ given by  formula \eqref{eq:ext-var2}.

Let us fix an arbitrary vector field $X \in \C^2_c(\R^N;\R^N)$ and an arbitrary function  $G\in \C^3(\R^N)$ satisfying $X\cdot \nu_\O=0$ and $G=g$ on $\partial\O$. 
According to Lemma \ref{lem:lien_inner_outer}, we have 
\begin{multline*}
\delta^2 AT_\e(u_\e,v_\e)[X,G]=\dd^2AT_\e(u_\e,v_\e)[X\cdot \nabla (u_\e-G), X\cdot \nabla v_\e]\\
+\dd AT_\e (u_\e,v_\e)[X\cdot \nabla (X\cdot \nabla (u_\e-G)),X\cdot \nabla (X\cdot \nabla v_\e)]\,.
\end{multline*}
Since $(u_\e,v_\e)=(g,1)$  and $X\cdot\nu_\O=0$ on $\partial\O$,  we  have $X\cdot\nabla (u_\e-G) =X\cdot \nabla v_\e=0$ on $\partial \O$. 
As a consequence, the functions $X\cdot \nabla (u_\e-G)$ and $X\cdot \nabla v_\e$ belong to $\C^{2}(\overline \O)$ and also vanish on $\partial\O$. Therefore, 
$X\cdot \nabla( X\cdot\nabla (u_\e-G))=X\cdot \nabla (X \cdot \nabla v_\e)=0$  on $\partial\O$. 
Next, $(u_\e,v_\e)$ being a critical point of $AT_\e$, we have 
$$\dd AT_\e (u_\e,v_\e)[X\cdot \nabla (X\cdot \nabla (u_\e-G)),X\cdot \nabla (X\cdot \nabla v_\e)=0\,.$$
Back to \eqref{stabinsect}, it follows that 
$$\delta^2AT_\e (u_\e,v_\e)[X,G]=\dd^2AT_\e(u_\e,v_\e)[ X\cdot \nabla (u_\e-G), X\cdot \nabla v_\e]\geq 0\,.$$
Passing now to the limit in the second inner  variation yields \eqref{secordmincritthm}, and the proof of Theorem \ref{thm:main_2} is now complete.
\end{proof}

\begin{remark}
{\rm In \cite{Cagnetti_Mora_Morini_2008,Bonacini_Morini_2015}, the authors explore  second order minimality conditions for the Mumford-Shah functional in the case where the jump set is regular enough. Such conditions could be derived in our context, taking care of the Dirichlet boundary data and thus of the fact that the jump set can charge the boundary. We do not develop this point here and refer to \cite[Theorem 3.6]{Cagnetti_Mora_Morini_2008} where the authors provide another expression for \(\delta^2MS(u)\) defined for smooth vector fields $X$ compactly supported  in~$\O$ (see Remark \ref{remgendeformations}). But we indicate that, as a consequence of Theorem \ref{thm:main_2}, it can be seen that, if \( (u_\e,v_\e)\in \A_g(\O) \) is a stable critical point of \(AT_\e\) such that, up to a subsequence, \( u_\e \to u_* \) in \(L^2(\O)\) and \eqref{eq:hypothesis} hold, then \(u_*\) satisfies the second order minimality condition for the Mumford-Shah functional derived in \cite{Cagnetti_Mora_Morini_2008,Bonacini_Morini_2015}) provided $\widehat J_{u_*}$ is sufficiently smooth. This follows by choosing \(X \in \C^\infty_c(\R^N;\R^N)\) of the form \( X=  \varphi\nu_{u_*} \circ \Pi_{\widehat{J}_{u_*}} \) in a neighborhood of \(\ \widehat{J}_{u_*}\) and satisfying \(\nu_{u_*} \cdot (DX \nu_{u_*} )=0 \) on \(\widehat{J}_{u_*}\), where \(\Pi_{\widehat{J}_{u_*}}\) denotes the nearest point projection onto \(\widehat{J}_{u_*}\) and $\varphi$ is an arbitrary smooth scalar function.}
\end{remark}

\section*{Appendix: First and second variations}
\setcounter{theorem}{0}
\setcounter{section}{1}

\renewcommand{\thesection}{\Alph{section}}

In this appendix we derive explicit expressions for outer and inner variations of the Ambrosio-Tortorelli and Mumford-Shah  functionals. 
First, we recall the expression of the first and second outer variations of $AT_\e$ are defined by  \eqref{1outvar} and  \eqref{2outvar}. 

\begin{lemma}\label{lem:outer-var}
Let $\O \subset \R^N$ be a bounded open set with Lipschitz boundary and $g \in H^{\frac12}(\partial\O)$. For all $(u,v) \in \mathcal A_g(\O)$ and all $(\phi,\psi) \in H_0^1(\O)\times [H_0^1(\O)\cap L^\infty(\O)]$,
\begin{eqnarray}
\dd AT_\e(u,v)[\phi,\psi]& =& 2\int_\O  \left[(\eta_\e+v^2)\nabla u \cdot \nabla \phi +\e \nabla v\cdot \nabla \psi+ |\nabla u|^2v\psi +\frac{(v-1)\psi}{4\e}\right] \dd x,\nonumber\\
\dd^2 AT_\e(u,v)[\phi,\psi]& = &  \int_\O  \left[(\eta_\e+v^2)|\nabla \phi|^2+4v\psi\nabla u \cdot \nabla \phi +\e |\nabla \psi|^2+|\nabla u|^2\psi^2+\frac{\psi^2}{4\e}\right] \dd x.\label{eq:ext-var2}
\end{eqnarray}
\end{lemma} 

The computations of inner variations rely on {one}-parameter groups of diffeomorphisms over $\overline\O$, or equivalently on their infinitesimal generators. More precisely, assuming that $\partial\O$ is of class $\C^{k+1}$ with $k\geq 1$, and given a vector field $X\in \C^k_c(\mathbb{R}^N;\mathbb{R}^N)$ satisfying $X\cdot\nu_\O =0$ on $\partial\O$, we consider the integral flow $\{\Phi_t\}_{t\in\R}$ of $X$  defined through the resolution of the ODE \eqref{eq:la_coulée} for every $x\in \R^N$.  Then $\Phi_0={\rm id}$ and the flow rule asserts that $\Phi_{t+s}=\Phi_t\circ\Phi_s$. Since $X\cdot\nu_\O =0$ on $\partial\O$, $\{\Phi_t\}_{t\in\R}$ is a {one}-parameter group of $\C^k$-diffeomorphism from $\O$ into itself, and from $\partial\O$ into itself. 

Given a (sufficiently smooth) boundary data $g$, we consider an arbitrary (smooth) extension $G$ of $g$ to $\overline\O$ to define a one-parameter family of deformations  $\{(u_t,v_t)\}_{t\in\R}\subset \mathcal{A}_g(\O)$ satisfying $(u_0,v_0)=(u,v)$ for a given pair $(u,v)\in \mathcal{A}_g(\O)$ by setting $u_t:=u\circ \Phi_{-t} -G\circ\Phi_{-t}+G$ and  $v_t:=v\circ\Phi_{-t}$. The first and second inner variations $\delta AT_\e$ and  $\delta^2 AT_\e$ of $AT_\e$ at $(u,v)$ are then defined by \eqref{eq:def_delta2}. 

\begin{remark}\label{remgendeformations}
{\rm
We emphasize that $\delta AT_\e(u,v)[X,G]$ and $\delta^2 AT_\e(u,v)[X,G]$ depend on both the vector field $X$ and the extension $G$ of the boundary data $g$, because the family of deformations $\{(u_t,v_t)\}_{t\in\R}$ depends on $X$ and $G$. It allows {one} to perform inner variations of the energy {\sl up to the boundary}. This type of deformations includes the more usual variation $\{(u\circ \Phi_{-t},v\circ{\Phi_{-t}}\}_{t\in \R}$ with $X$ compactly supported in~$\O$. Indeed, in this case we may choose an extension $G$ supported in a small neighborhood of $\partial\O$ in such a way that ${\rm supp}\,G\cap{\rm supp}\,X=\emptyset$. Then $G\circ \Phi_{-t}=G$, and thus $u_t=u\circ \Phi_{-t}$. }
\end{remark}

If the pair $(u,v)$ and $\partial \O$ are smooth enough, one can compute the first and second inner variations of $AT_\e$ at $(u,v)$ using the Taylor expansion of $(u_t,v_t)$ with respect to the parameter $t$. One may for instance follow the general setting of \cite[Lemma 2.2 and Corollary 2.3]{Le_Sternberg_2019}.

\begin{lemma}\label{lem:lien_inner_outer}
Let $\O \subset \R^N$ be a bounded open set with  boundary of class $\C^2$, $g \in \C^{2}(\partial\O)$ and $(u,v)\in\mathcal{A}_g(\O)\cap[\C^2(\overline\O)]^2$.
\begin{enumerate}
\item[(i)]  Then for every {vector field }$X\in \C^1_c(\R^N;\R^N)$ and {every extension} $G\in \C^2(\R^N)$ satisfying $X\cdot\nu_\O=0$ and $G=g$ on $\partial\O$,
$$\delta AT_\e (u,v)[X,G]=-\dd AT_\e(u,v)[X\cdot \nabla (u-G),X\cdot\nabla v].$$

\vskip3pt

\item[(ii)] If {further} $\partial\O$ is of class $\C^3$, $g \in \C^{3}(\partial\O)$, and  $(u,v)\in\mathcal{A}_g(\O)\cap[\C^3(\overline\O)]^2$, then for every {vector field } $X\in \C^2_c(\R^N;\R^N)$ and {every extension} $G\in \C^3(\R^N)$ satisfying $X\cdot\nu_\O=0$ and $G=g$ on~$\partial\O$,
\begin{align*}
\delta^2 AT_\e(u,v)[X,G]=&\,\dd^2AT_\e(u,v)[X\cdot \nabla (u-G),X\cdot\nabla v]\\
&\qquad+\dd AT_\e(u,v)\big[X\cdot\nabla(X\cdot\nabla(u-G)),X\cdot\nabla(X\cdot\nabla v)\big]\,.
\end{align*}
\end{enumerate}
\end{lemma}

\begin{proof}
Define $\bar u_t:=u\circ \Phi_{-t}-G\circ\Phi_{-t}$. Since $(u,v)$ and $G$ belong to $\C^2(\overline\O)$, we can differentiate $(u_t,v_t)$ with respect to $t$ and use \eqref{eq:la_coulée} with the flow rule $\Phi_{t+s}=\Phi_t\circ\Phi_s$ to find 
$$(\dot u_t,\dot v_t):=\frac{\dd}{\dd t}(u_t,v_t)=\frac{\dd}{\dd s}\Bigl|_{s=0} (\bar u_{t+s},v_{t+s})= \big(-X\cdot\nabla\bar u_t,  -X\cdot\nabla v_t \big) \in [\C^1(\overline\O)]^2\,.$$
In particular, we have
\begin{equation}\label{calcdotuv}
(\dot u_0,\dot v_0)=\big(-X\cdot\nabla (u-G), -X\cdot\nabla v\big)\,.
\end{equation}
If $\partial\O$ is of class $\mathscr C^3$, $g \in \mathscr C^{3}(\partial\O)$, and  $(u,v)\in\mathcal{A}_g(\O)\cap[\mathscr C^3(\overline\O)]^2$, then we can differentiate $(\dot u_t,\dot v_t)$ with respect to $t$ to obtain  
$$\ddot u_t:=\frac{\dd}{\dd t}\dot u_t=-X\cdot\nabla \dot u_t =X\cdot\nabla\big(X\cdot\nabla \bar u_t\big)\quad\text{and}\quad \ddot v_t:=\frac{\dd}{\dd t}\dot v_t=-X\cdot\nabla \dot v_t =X\cdot\nabla\big(X\cdot\nabla v_t\big)\,.$$
Hence, 
\begin{equation}\label{calcddotu}
\ddot u_0=X\cdot\nabla\big(X\cdot\nabla(u-G)\big)\quad\text{and}\quad \ddot v_0=X\cdot\nabla\big(X\cdot\nabla v\big)\,.
\end{equation}
Next, elementary computations yield
$$\frac{\dd}{\dd t}AT_\e(u_t,v_t)=\dd AT_\e(u_t,v_t)[\dot u_t,\dot v_t]\,, $$
and 
$$\frac{\dd^2}{\dd t^2}AT_\e(u_t,v_t)=\dd^2AT_\e(u_t,v_t)[\dot u_t,\dot v_t]+\dd AT_\e(u_t,v_t)[\ddot u_t,\ddot v_t]\,,  $$
so that the conclusion follows from \eqref{calcdotuv}-\eqref{calcddotu} evaluating those derivatives at $t=0$. 
\end{proof}

In case the pair $(u,v)$ only belongs to the energy space $\mathcal{A}_g(\O)$, we can compute the first and second inner variations of $AT_\e$ by making the change variables  $y=\Phi_t(x)$ in the integrals defining $AT_\e(u_t,v_t)$. Then one expands the result with respect to $t$ using a Taylor expansion of $\Phi_t$.  If $X\in \C^2_c(\R^N;\R^N)$, the  second order Taylor expansion  near $t=0$ of the flow map $\Phi_t$ induced by $X$ is given by 
\begin{equation}\label{expanphit}
\Phi_t={\rm id} + t X+\frac{t^2}{2}Y+o(t^2)\,, 
\end{equation}
where $Y\in \C^1_c(\R^N;\R^N)$ denotes the vector field $Y:=(DX)X$\,, 
$DX$ being the Jacobian matrix of $X$ (i.e., $(DX)_{ij}=\partial_j X_i$ with $i$ the row index and $j$ the column index), and $o(s)$ denotes a quantity satisfying $o(s)/s\to 0$ as $s\to 0$ uniformly with respect to $x\in\R^N$.

\begin{lemma}\label{lem:expressions_inner}
Let $\O \subset \R^N$ be a bounded open set with  boundary of class $\C^2$, $g \in \C^{2}(\partial\O)$ and $(u,v)\in\mathcal{A}_g(\O)$.
\begin{enumerate}
\item[(i)] Then for every {vector field } $X\in \C^1_c(\R^N;\R^N)$ and {every extension} $G\in \C^2(\R^N)$ satisfying $X\cdot\nu_\O=0$ and $G=g$ on $\partial\O$,
\begin{multline}\label{eq:1st_inner_variation}
\delta AT_\e (u,v)[X,G]=\int_\O (\eta_\e+v^2) \Big[|\nabla u|^2 {\rm Id} -2 \nabla u \otimes \nabla u\Big]:D X\dd x\\
+ \int_\O\left[ \left(\e|\nabla v|^2 +\frac{(v-1)^2}{4\e} \right){\rm Id}-2\e \nabla v \otimes \nabla v \right]:D X \dd x +2\int_\O(\eta_\e+v^2)\nabla u\cdot\nabla(X\cdot\nabla G)\,\dd x\,.
\end{multline}
\vskip3pt

\item[(ii)] If {further} $\partial\O$ is of class $\C^3$, $g \in \C^{3}(\partial\O)$, then for every {vector field } $X\in \C^2_c(\R^N;\R^N)$ and {every extension} $G\in \C^3(\R^N)$ satisfying $X\cdot\nu_\O=0$ and $G=g$ on $\partial\O$,
\begin{align}\label{eq:2nd_inner_variation}
\delta^2 AT_\e (u,v)[X,G]=& \int_\O(\eta_\e+v^2)\bigg( |\nabla u|^2{\rm Id}-2(\nabla u\otimes\nabla u)\bigg):DY \dd x \nonumber \\
&+\int_\O(\eta_\e+v^2)\bigg\{|\nabla u|^2\big(({\rm div}X)^2-{\rm tr}((DX)^2)\big) -4\big((\nabla u\otimes\nabla u):DX\big){\rm div} X \nonumber \\
&\qquad\qquad\qquad\qquad\qquad+4(\nabla u\otimes\nabla u):(DX)^2 +2 |DX^T\nabla u|^2\bigg\}\,\dd x \nonumber \\
&+4\int_\O(\eta_\e+v^2)\big[(\nabla u\cdot\nabla (X\cdot \nabla G)){\rm div}X\nonumber\\
&\hspace{4cm}-\big(\nabla u\otimes\nabla (X\cdot\nabla G)\big):\big(DX+(DX)^T\big)\big]\,\dd x \nonumber \\
& +2\int_\O(\eta_\e+v^2)\nabla  u\cdot\nabla (X\cdot\nabla (X\cdot\nabla G))\,\dd x +2\int_\O(\eta_\e+v^2)|\nabla(X\cdot\nabla G)|^2\,\dd x\nonumber\\
&+ \int_\O\bigg[\bigg(\e|\nabla v|^2+\frac{(v-1)^2}{4\e}\bigg){\rm Id}-2\e\nabla v\otimes \nabla v\bigg]:DY \dd x \nonumber \\
&+\int_\O\bigg(\e|\nabla v|^2+\frac{(v-1)^2}{4\e}\bigg)\big(({\rm div}X)^2-{\rm tr}((DX)^2)\big) \dd x \nonumber  \\
&-4\int_\O\e\big((\nabla v\otimes\nabla v):DX\big){\rm div} X\,\dd x+4\int_\O\e(\nabla v\otimes\nabla v):(DX)^2\,\dd x \nonumber \\
&+2\int_\O\e |DX^T\nabla v|^2\,\dd x\,,
\end{align}
with $Y:=(DX)X$. 
\end{enumerate}
\end{lemma}

\begin{remark}\label{Rem:dependance_G}
{\rm From \eqref{eq:1st_inner_variation} we see that if \((u,v)\in \A_g(\O)\) is a critical point of \(AT_\e\) in the sense that \( \dd AT_\e(u,v)[\phi,\psi]=0\) for all admissible \( (\phi,\psi)\), then \(\delta AT_\e (u,v)[X,G]\) is independent of the extension \(G\). Indeed, an integration by parts and  the first equation in \eqref{eq:PDE} show that, in this case,
$$\int_\O (\eta_\e+v^2) \nabla u \cdot \nabla (X\cdot \nabla G) \dd x=(\eta_\e+1)\int_{\p \O} \p_\nu u (X \cdot \nabla_\tau g) \dd \mathcal{H}^{N-1}\,.$$
since \(X \cdot \nu_\O =0\) on $\partial\O$ (\(\nabla_\tau \) is the tangential gradient on \(\p \O\)).  For the second inner variation \eqref{eq:2nd_inner_variation}, even if \( (u,v)\in \A_g(\O)\) satisfies \(\dd AT_\e(u,v)[\phi,\psi]=0\) for \( (\phi,\psi)\in [\C^\infty_c(\O)]^2\), the expression \(\delta^2 AT_\e(u,v)[X,G]\)  does depend on  the extension \(G\) and not only on the boundary data \(g\) {because of} the terms 
\[\int_\O(\eta_\e+v^2)\big[(\nabla u\cdot\nabla (X\cdot \nabla G)){\rm Id}-2\nabla u\otimes\nabla (X\cdot\nabla G)\big]:DX\,\dd x +\int_\O(\eta_\e+v^2)|\nabla(X\cdot\nabla G)|^2\,\dd x\,.\]
If we take \(X \in \C^\infty_c(\O;\R^N)\) and \(G \in \C^2(\R^N)\) an extension of \(g\) such that \( \supp G \cap \supp X =\emptyset\), and if we assume {\( (u,v) \in \A_g(\O)\) to be} a critical of \(AT_\e\), then the expression of the second-inner variation \eqref{eq:2nd_inner_variation} simplifies. Indeed the terms that contain \(Y=(DX)X\) disappear, since by regularity we have \(\delta AT_\e(u,v)[Y,G]=0\), and all terms containing \(G\) disappear. In this case, we end up with
\begin{align*}
\delta^2 AT_\e (u,v)[X,G]= &\int_\O(\eta_\e+v^2)\bigg\{\big(({\rm div}X)^2-{\rm tr}((DX)^2)\big) -4((\nabla u\otimes\nabla u):DX){\rm div} X \nonumber \\
&\qquad\qquad\qquad\qquad\qquad+4(\nabla u\otimes\nabla u):(DX)^2 +2 |DX^T\nabla u|^2\bigg\}\,\dd x \nonumber \\
&+\int_\O\bigg(\e|\nabla v|^2+\frac{(v-1)^2}{4\e}\bigg)\big(({\rm div}X)^2-{\rm tr}((DX)^2)\big) \dd x \nonumber  \\
&-4\int_\O\e((\nabla v\otimes\nabla v):DX){\rm div} X\,\dd x+4\int_\O\e(\nabla v\otimes\nabla v):(DX)^2\,\dd x \nonumber \\
&+2\int_\O\e |DX^T\nabla v|^2\,\dd x\,.
\end{align*}}
\end{remark}

\begin{proof}[Proof of Lemma \ref{lem:expressions_inner}]
For simplicity, we assume that $\partial\O$ is of class $\C^3$, $g \in \C^{3}(\partial\O)$, and we observe that the computation of $\delta AT_\e$  below only requires $\C^2$ regularity. We fix $X\in \C^2_c(\R^N;\R^N)$ and $G\in \C^3(\overline\O)$  satisfying $X\cdot\nu=0$ and $G=g$ on $\partial\O$. We set $\widehat u_t:=u\circ\Phi_{-t}$ and $G_t:=G\circ\Phi_{-t}$. 
Since $u_t=\widehat u_t-(G_t-G)$, we have 
\begin{align}
\nonumber AT_\e(u_t,v_t)&=AT_\e(\widehat u_t,v_t)-2\int_\O(\eta_\e+v_t^2)\nabla\widehat u_t\cdot\nabla (G_t-G)\,\dd x+ \int_\O(\eta_\e+v_t^2)|\nabla (G_t-G)|^2\,\dd x\\
\label{decompenergvarinner} &=:\mathcal{A}(t)+\mathcal{B}(t)+\mathcal{C}(t)\,.
\end{align}
We aim to compute the first and second derivatives at $t=0$ of $\mathcal{A}$, $\mathcal{B}$, and $\mathcal{C}$, starting with $\mathcal{A}$. 

By the chain rule in Sobolev spaces, we have
\begin{equation}\label{devgraduhatt}
\nabla \widehat u_t=[D\Phi_{-t}]^{T}\nabla u(\Phi_{-t})=[D\Phi_t(\Phi_{-t})]^{-T}\nabla u(\Phi_{-t})\,.
\end{equation}
On the other hand, in view of \eqref{expanphit}, we have 
\begin{equation}\label{expDphi-1}
[D\Phi_t]^{-1}= I-tDX-\frac{t^2}{2}Y+t^2(DX)^2+o(t^2)\,,
\end{equation}
and 
\begin{equation}\label{expandjacob}
{\rm det}(D\Phi_t)=1+t{\rm div}X+\frac{t^2}{2}\big[{\rm div}Y+({\rm div}X)^2-{\rm tr}((DX)^2)\big] +o(t^2)\,.
\end{equation}
Using the change of variables $x=\Phi_t(y)$, classical computations (see e.g. \cite{Le_Sternberg_2019}) yield

\begin{multline}\label{aprime0}
\mathcal{A}^\prime(0)=\int_\O (\eta_\e+v^2) \left[|\nabla u|^2 {\rm Id} -2 \nabla u \otimes \nabla u\right]:D X\dd x\\
+ \int_\O\left[ \left(\e|\nabla v|^2 +\frac{(v-1)^2}{4\e} \right){\rm Id}-2\e \nabla v \otimes \nabla v \right]:D X \dd x\,, 
\end{multline}
and 
\begin{align}
\nonumber \mathcal{A}^{\prime\prime}(0)=& \int_\O(\eta_\e+v^2)\bigg\{ |\nabla u|^2\big({\rm div}Y+({\rm div}X)^2-{\rm tr}((DX)^2)\big) -4((\nabla u\otimes\nabla u):DX){\rm div} X\\
\nonumber &\qquad\qquad\qquad\qquad\qquad\qquad\qquad\qquad-2(\nabla u\otimes\nabla u):(DY-2(DX)^2) +2 |DX^T\nabla u|^2\bigg\}\,\dd x\\
\nonumber &+\int_\O\bigg(\e|\nabla v|^2+\frac{(v-1)^2}{4\e}\bigg)\big({\rm div}Y+({\rm div}X)^2-{\rm tr}((DX)^2)\big) -4\e((\nabla v\otimes\nabla v):DX){\rm div} X\,\dd x\\
\label{aprimeprime0} &\qquad\qquad\qquad\qquad-2\int_\O\e(\nabla v\otimes\nabla v):(DY-2(DX)^2)\,\dd x +2\int_\O \e|DX^T\nabla v|^2\,\dd x\,.
\end{align}

Next we compute the derivatives of $\mathcal{B}$. To this purpose, we observe that 
$$\nabla \widehat u_{t+s}=[D\Phi_s(\Phi_{-s})]^{-T} \nabla \widehat u_t(\Phi_{-s})\,,\quad \nabla G_{t+s}=[D\Phi_s(\Phi_{-s})]^{-T}\nabla G_t(\Phi_{-s})\,,$$
and by \eqref{expDphi-1}, 
$$[D\Phi_s]^{-T}\nabla G_t-\nabla G(\Phi_s)=\nabla (G_t-G)-s(DX)^T \nabla (G_t-G)-s\nabla (X\cdot \nabla G)+o(s)\,.$$
Using the change of variables $x=\Phi_s(y)$ and \eqref{expDphi-1}--\eqref{expandjacob} again, we obtain 
\begin{multline*}
\int_\O(\eta_\e+v_{t+s}^2)\nabla\widehat u_{t+s}\cdot\nabla (G_{t+s}-G)\,\dd x\\
=\int_\O(\eta_\e+v_t^2)\big([D\Phi_s]^{-T}\nabla \widehat u_t\big)\cdot \big([D\Phi_s]^{-T}\nabla G_t -\nabla G(\Phi_s)\big){\rm det}(D\Phi_s)\,\dd y\\
=\int_\O(\eta_\e+v_{t}^2)\nabla\widehat u_{t}\cdot\nabla (G_{t}-G)\,\dd x-s\int_\O(\eta_\e+v^2)\nabla \widehat u_t\cdot\nabla (X\cdot\nabla G)\,\dd x\\
+s\int_\O(\eta_\e+v_{t}^2)\big[(\nabla\widehat u_{t}\cdot\nabla (G_{t}-G)){\rm div} X- (\nabla\widehat u_t\otimes\nabla(G_t-G)):\big(DX+(DX)^T\big)\big]\,\dd x+o(s)\,.
\end{multline*}
Consequently,
\begin{multline*}
\mathcal{B}^\prime(t)=2\int_\O(\eta_\e+v_t^2)\nabla \widehat u_t\cdot\nabla (X\cdot\nabla G)\,\dd x\\
-2\int_\O(\eta_\e+v_{t}^2)\big[(\nabla\widehat u_{t}\cdot\nabla (G_{t}-G)){\rm div} X- (\nabla\widehat u_t\otimes\nabla(G_t-G)):\big(DX+(DX)^T\big)\big]\,\dd x\, ,
\end{multline*}
and in particular,  
\begin{equation}\label{bprime0}
\mathcal{B}^\prime(0)=2\int_\O(\eta_\e+v^2)\nabla  u\cdot\nabla (X\cdot\nabla G)\,\dd x\,.
\end{equation}
To compute $\mathcal{B}^{\prime\prime}(0)$, we write $\mathcal{B}^\prime(t)=:I(t)-II(t)$, and we set for simplicity $H:=X\cdot \nabla G$.  
Since 
$$\nabla H(\Phi_t)=\nabla H +t\nabla(X\cdot \nabla H) -t(DX)^T \nabla H+o(t) \,,$$
we can change variables $y=\Phi_t(x)$ and use again \eqref{devgraduhatt}-\eqref{expDphi-1}-\eqref{expandjacob} to find
\begin{multline*}
I(t)= 2\int_\O(\eta_\e+v^2)\nabla u\cdot\nabla H\,\dd x+2t\int_\O(\eta_\e+v^2)\nabla  u\cdot\nabla (X\cdot\nabla H)\,\dd x\\
+2t\int_\O(\eta_\e+v^2)\big[(\nabla u\cdot\nabla H){\rm div} X-(\nabla u\otimes\nabla H):\big(DX+(DX)^T\big)\big]\,\dd x+o(t)\,.
\end{multline*}
Since $G_t-G=-tH+o(t)$, we easily infer that
$$II(t)=-2t\int_\O(\eta_\e+v^2)\big[(\nabla u\cdot\nabla H){\rm div}X-(\nabla u\otimes\nabla H):\big(DX+(DX)^T\big)\big]\,\dd x+o(t)\,, $$ 
and consequently,
\begin{multline}\label{bprimeprime0}
\mathcal{B}^{\prime\prime}(0)=2\int_\O(\eta_\e+v^2)\nabla  u\cdot\nabla (X\cdot\nabla H)\,\dd x\\
+4\int_\O(\eta_\e+v^2)\big[(\nabla u\cdot\nabla H){\rm div} X-(\nabla u\otimes\nabla H):\big(DX+(DX)^T\big)\big]\,\dd x\,.
\end{multline}
Similarly, 
$$\mathcal{C}(t)=t^2\int_\O(\eta_\e+v^2)|\nabla H|^2\,\dd x+o(t^2)\,, $$
so that 
\begin{equation}\label{derivatC0}
\mathcal{C}^{\prime}(0)=0\quad\text{and}\quad \mathcal{C}^{\prime\prime}(0)=2\int_\O(\eta_\e+v^2)|\nabla H|^2\,\dd x\,.
\end{equation}
In view of \eqref{decompenergvarinner}, gathering \eqref{aprime0}-\eqref{bprime0}-\eqref{derivatC0} or  \eqref{aprimeprime0}-\eqref{bprimeprime0}-\eqref{derivatC0}  leads to the announced formula for $\delta AT_\e (u,v)[X,G]$ and $\delta^2 AT_\e (u,v)[X,G]$ respectively. 
\end{proof}

Similar computations together with the well known first and second inner variations of a countably $\HH^{N-1}$-rectifiable set (see e.g.\ \cite[Chapter 2, Section 9]{Simon_1983}) lead to explicit expressions for the first and second inner variations of the Mumford-Shah functional. 

\begin{lemma}\label{lem:explicit_inner_stab2}
Let $\O\subset \R^N$ be a bounded open set of class $\C^{2}$, $g \in \mathscr C^2(\partial\O)$ and \( u \in SBV^2(\O) \). 
\begin{itemize}
\item[(i)] Then for all {vector field} \(X \in \C^1_c(\R^N;\R^N)\) {and every extension $G \in \mathscr C^2(\R^N)$ with $X \cdot \nu_\O=0$ and $G=g$ on $\partial\O$}, we have
\begin{eqnarray}\label{eq:1146}
\delta MS(u)[X,G] &=& \int_\O\left( |\nabla u|^2\rm{Id}-2\nabla u\otimes \nabla u\right):D X \dd x\nonumber\\
&&+ \int_{\widehat J_u} \dive^{\widehat J_u} X\dd\HH^{N-1}+2\int_\O \nabla u\cdot \nabla (X\cdot \nabla G) \dd x\, .
\end{eqnarray}
\item[ii)] If further $\O$ is of class $\C^3$ and $g \in \C^3(\partial \O)$, then for all {vector field} \(X \in \C^2_c(\R^N;\R^N)\) {and every extension $G \in \mathscr C^3(\R^N)$ with $X \cdot \nu_\O=0$ and $G=g$ on $\partial\O$}, we have
\begin{eqnarray*}
\delta^2 MS(u)[X,G]&  = &  \int_\O \left( |\nabla u|^2\rm{Id}-2\nabla u\otimes \nabla u \right):DY \dd x +\int_{\widehat{J}_u} \dive^{\widehat{J}_u}Y \dd \HH^{N-1}  \\
&& +\int_\O \left[|\nabla u|^2\left( (\dive X)^2-\tr ((DX)^2) \right) -4\big((\nabla u\otimes \nabla u) :DX\big)\dive X\right] \dd x  \\
&&+\int_\O \left[4(\nabla u\otimes \nabla u): (DX)^2+2 |DX^T\nabla u|^2\right] \dd x  \\
&&+\int_{\widehat J_u}\left[(\dive^{\widehat J_u} X)^2+\sum_{i=1}^{N-1}|(\partial_{\tau_i}X)^\perp|^2- \sum_{i,j=1}^{N-1} \left( \tau_i\cdot \partial_{\tau_j} X \right) \left( \tau_j \cdot \partial_{\tau_i} X \right)\right] \dd \HH^{N-1} \\
&&+2 \int_\O \nabla u\cdot \nabla (X\cdot \nabla (X\cdot G)) \dd x +2\int_\O |\nabla (X\cdot \nabla G)|^2 \dd x  \\
&&+4 \int_\O\Big[ \big(\nabla u\cdot \nabla (X\cdot \nabla G)\big) {\rm div}X-(\nabla u\otimes \nabla (X\cdot \nabla G)) :(DX+(DX)^T)\Big] \dd x\,,
\end{eqnarray*} 
where \((\tau_1,\dots,\tau_{N-1})\) is a basis of the tangent plane to \(\widehat J_u\) at a given point \(x\in \widehat J_u\). 
\end{itemize}
\end{lemma}

\begin{proof}
The second inner variation of the part \(\int_\O |\nabla u|^2 \dd x\) is computed exactly as in the proof of Lemma \ref{lem:expressions_inner}, {recalling that the chain rule still holds for the approximate gradient} \(\nabla (u\circ \Phi_{-t})=[D \Phi_{-t}]^T\nabla u(\Phi_{-t})\). For the singular part of the energy, {we use that $\mathcal{H}^{N-1}(\widehat J_{u_t})=\mathcal{H}^{N-1}( J_{\hat{u}_t})$ where \(\hat{u}_t=\hat u \circ \Phi_{-t}\) and $\hat u=u {\bf 1}_\O + G {\bf 1}_{\R^N\setminus \O}$}. The second variation of such a functional is computed with the area formula as in \cite[Chapter 2]{Simon_1983} {together with the following geometric formulas 
\begin{multline*}
(\dive X)^2-\tr[(\nabla X)^2]-2 ((\nu_u \otimes \nu_u):D X) \dive X+2 (\nu_u \otimes \nu_u): (D X)^2+|D X^T \nu_u|^2 \\
= (\dive^{J_u} X)^2+\sum_{i=1}^{N-1}|(\partial_{\tau_i}X)^\perp|^2- \sum_{i,j=1}^{N-1} \left( \tau_i\cdot \partial_{\tau_j} X \right) \left( \tau_j \cdot \partial_{\tau_i} X \right) +\left( (\nu_u \otimes \nu_u):\nabla X\right)^2
\end{multline*}
stated in the proof of Theorem 1.1 p.\ 1851--1852 in \cite{Le_2011}.}
\end{proof}

\begin{remark}\label{remfinale}
{\rm {As in Remark \ref{Rem:dependance_G}, \(\delta MS(u)[X,G]\) is independent of} the extension \(G\) when \(u\) is a critical point for the outer variations of \(MS\), {while \(\delta^2 MS(u)[X,G]\) does depend} on the extension \(G\) in general. {If \(u\) satisfies \(\delta MS(u)[X,G]=0\), then for all \(X\in \C^\infty_c(\O,\R^N)\) with \( \supp G \cap \supp X = \emptyset\), then the formula of the second inner variation reduces to}
\begin{eqnarray*}
\delta^2 MS(u)[X,G] &= & \int_\O |\nabla u|^2\left( (\dive X)^2-\tr ((DX)^2) \right) -4((\nabla u\otimes \nabla u) :D X) \dive X \dd x \nonumber \\
&+&\int_\O 4\nabla u\otimes \nabla u: (DX)^2+2 |DX^T\nabla u|^2 \dd x \nonumber \\
&+& \int_{\widehat J_u}\left[(\dive^{\widehat J_u} X)^2+\sum_{i=1}^{N-1}|(\partial_{\tau_i}X)^\perp|^2- \sum_{i,j=1}^{N-1} \left( \tau_i\cdot \partial_{\tau_j} X \right) \left( \tau_j \cdot \partial_{\tau_i} X \right)\right] \dd \HH^{N-1}\,. \nonumber\\
\end{eqnarray*}}
\end{remark}

\bibliographystyle{abbrv}

\end{document}